\documentclass[11pt, reqno]{amsart}
\usepackage{amssymb,verbatim,amscd,amsmath}
\usepackage{mathrsfs}
\usepackage{braket}
\usepackage[pagebackref,colorlinks=true,citecolor=blue,linkcolor=magenta]{hyperref}
\usepackage{array}
\usepackage{mathtools}
\usepackage{verbatim}
\usepackage{caption}
\usepackage{subcaption}
\usepackage{pdfsync}
\usepackage{tabu}
\usepackage{fullpage}
\usepackage{color}
\usepackage[svgnames]{xcolor}
\usepackage{enumerate}
\usepackage{pst-node,pst-text,pst-3d,pstricks}
\usepackage[dvips,all]{xy}
\usepackage{graphics}
\usepackage{tikz}
\usepackage{pst-plot}
\usetikzlibrary{patterns}
\usetikzlibrary{calc}
\usepackage{cleveref}
\usetikzlibrary{matrix}
\usetikzlibrary{arrows}
\usetikzlibrary{decorations.pathmorphing}
\usepackage{graphicx,supertabular,paralist}
\usetikzlibrary{decorations.pathreplacing, decorations.markings}
\tikzset{->-/.style={decoration={markings,mark=at position #1 with {\arrow{>}}},postaction={decorate}}}

\usepackage{tikz-cd}
  \usepackage{paralist}

\numberwithin{equation}{section}
\newtheorem{theorem}{Theorem}[section]
\newtheorem{lemma}[theorem]{Lemma}
\newtheorem{proposition}[theorem]{Proposition}
\newtheorem{corollary}[theorem]{Corollary}

\theoremstyle{definition}
\newtheorem{definition}[theorem]{Definition}

\newtheorem{observation}[theorem]{Observation}
\newtheorem{def-prop}[theorem]{Definition-Proposition}
\newtheorem{remark}[theorem]{Remark}
\newtheorem{example}[theorem]{Example}

\newtheorem{notation}[theorem]{Notation}
\newtheorem{question}[theorem]{Question}

\DeclareMathOperator{\supp}{supp}

\DeclareMathOperator{\im}{Im}

\DeclareMathOperator{\ini}{in}
\DeclareMathOperator{\spa}{span}

\definecolor{navyblue}{rgb}{0.0, 0.0, 0.5}
\definecolor{darkred}{rgb}{0.55, 0.0, 0.0}

\newcommand{\MB}{{\mathcal B}}

\newcommand{\NN}{{\mathbb N}}

\newcommand{\KK}{{\mathbb K}}

\begin{document}

\title{Multi-Rees Algebras of Strongly Stable Ideals}

\author{Selvi Kara}
\address{Department of Mathematics, University of Utah, 155 1400 E, Salt Lake City, UT 84112}
\email{selvi@math.utah.edu}
\urladdr{}

\author{Kuei-Nuan Lin}
\address{Department of Mathematics,
Penn State University, Greater Allegheny campus\\ 4000 University Dr, McKeesport, PA 15132, USA}
\email{kul20@psu.edu}
\urladdr{}

\author{Gabriel Sosa Castillo}
\address{Department of Mathematics, Colgate University, 3 Oak Drive Hamilton, NY 13346, USA}
\email{gsosacastillo@colgate.edu }
\urladdr{}


\begin{abstract} 

   We prove that  the multi-Rees algebra $\mathcal{R}(I_1 \oplus \cdots \oplus I_r)$  of a collection of strongly stable ideals $I_1, \ldots, I_r$ is of fiber type. In particular, we provide a Gr\"obner basis for its defining ideal as a union of a Gr\"obner basis for its special fiber and binomial syzygies.  We also study the Koszulness of $\mathcal{R}(I_1 \oplus \cdots \oplus I_r)$ based on parameters associated to the collection. Furthermore, we establish a quadratic Gr\"obner basis of the defining ideal of $\mathcal{R}(I_1 \oplus I_2)$ where each of the strongly stable ideals has two quadric Borel generators. As a consequence, we conclude that this multi-Rees algebra is Koszul.

\end{abstract}

\maketitle


\section*{Introduction}

Rees algebras are central and important objects in commutative algebra and algebraic geometry (see \cite{vasconcelos1994arithmetic}) with numerous appearances in various other fields including elimination theory, geometric modeling and chemical reaction networks (see \cite{CWL}, \cite{Cox} and \cite{CoxLinSosa}). 
From a geometrical point of view, the Rees algebra $\mathcal{R} (I)$ of an ideal $I$ in a polynomial ring $S=\mathbb{K}[x_1,\ldots, x_n]$ is the homogeneous coordinate ring of two fundamental objects: the blowup of a projective space along the subscheme defined by $I$ and the graph of a rational map between projective spaces defined by the generators of $I$.  Naturally, the multi-Rees algebra of a collection of ideals $I_1, \ldots, I_r$ of $S$ is the homogeneous coordinate ring of the blowup along the subschemes defined by the ideals.  The multi-Rees algebra of $I_1, \ldots, I_r$ is also the Rees algebra of the module $I_1 \oplus \cdots \oplus I_r$  and it is defined as the multigraded $S$-algebra:
$$\mathcal{R} (I_1 \oplus \cdots \oplus I_r)= \bigoplus_{(a_1, \ldots, a_r) \in \mathbb{Z}_{\geq 0}^n} I_1^{a_1} \cdots I_r^{a_r}t_1^{a_1}\cdots t_r^{a_r} \subseteq S[t_1,\ldots, t_r]$$
for auxiliary variables $t_1, \ldots, t_r$.  A related object is the special fiber ring which is associated to the image of the blowup map. We denote the special fiber ring  by $\mathcal{F} (I_1 \oplus \cdots \oplus I_r)$ and it is equal to $\mathcal{R} (I_1 \oplus \cdots \oplus I_r) \otimes_{S}\mathbb{K}$.

A fundamental problem in the study of Rees and multi-Rees algebras is finding the implicit equations of the defining ideals $\mathcal{L}$ and $\mathcal{K}$ of the multi-Rees algebra and its special fiber, respectively, such that $\mathcal{R} (I_1 \oplus \cdots \oplus I_r) \cong S[T_1,\ldots, T_N]/\mathcal{L}$  and $\mathcal{F} (I_1 \oplus \cdots \oplus I_r) \cong \mathbb{K}[T_1,\ldots, T_N]/\mathcal{K}$  where $N$ is the total number of minimal generators of $I_1, \ldots, I_r$. This challenging problem has been studied extensively for Rees algebras ($r=1$), however, it is still open for many classes of ideals. In recent years, there has been some progress toward addressing this problem for the multi-Rees algebras \cite{dipasquale2020koszul, Babak, LinPolini, Sos} albeit the increased complexity of this setting.  In the light of all the established cases, a complete solution can be provided if the ideals in question have specific structures.

In this paper, we study the multi-Rees algebras of strongly stable ideals (also called Borel ideals in literature). These ideals are of special importance in computational commutative algebra  due to their nice combinatorial structure; in characteristic zero, these ideals coincide with Borel-fixed ideals and they occur as generic initial ideals (\cite{BS2}, \cite{Ga}).

In addition to finding defining equations of the multi-Rees algebras  of strongly stable ideals and their special fiber rings, we are interested in detecting when these algebras are Koszul.  An arbitrary graded ring $R$ over a field $R_0 = \mathbb{K}$ is Koszul if the residue field $R/R_{+} \cong \mathbb{K}$ has a linear resolution over $R$. A common approach to prove an algebra $R$ is Koszul is to show that it is $G$-quadratic, that is, $R \cong S/J$ is the quotient of a polynomial ring $S$ by an ideal $J \subset S$ such that $J$ has a Gr\"obner basis of quadrics with respect to some monomial order (see  \cite[Theorem~6.7]{EH}, \cite{F}).  Many of the classically studied rings in commutative algebra and algebraic geometry are Koszul and these algebras have good homological features; as  the authors of \cite{conca2013koszul} put  ``a homological life is worth living in a Koszul algebra." It is well-known that Koszulness of $\mathcal{R}(I)$ implies that $I$ has linear powers, i.e., every power of $I$ has a linear resolution \cite[Corollary 3.6]{Blum}. The multigraded version of this result is proved by Bruns and Conca in \cite[Theorem 3.4]{BC}; when $\mathcal{R} (I_1 \oplus \cdots \oplus I_r)$ is Koszul, then products of $I_1,\ldots, I_r$ have linear resolutions.

Both problems of interest, implicit equations for the defining ideals and Koszulness, can be answered simultaneously by identifying  an explicit quadric Gr\"obner basis for the defining ideal of the multi-Rees algebra or its special fiber ring. However, it is difficult to find such Gr\"obner basis in general.  One of the earlier works on the multi-Rees algebras of strongly stable ideals addressing both problems is due to Lin and Polini \cite{LinPolini} in which they consider powers of maximal ideals. In particular, they provide quadric Gr\"obner bases of $\mathcal{L}$ and $\mathcal{K}$ in \cite[Theorem 2.4]{LinPolini} and further show that the associated multi-Rees algebra and  its special fiber ring are both Koszul domain. Their work is extended to a more general class of principal strongly stable ideals by Sosa in \cite[Theorem 2.2]{Sos}.

A closer look into the literature on the Rees algebras of strongly stable ideals suggest a systematic approach for the multi-Rees setting. Namely, one can study the multi-Rees algebras of strongly stable ideals $I_1, \ldots, I_r$  by considering three parameters: the number of ideals $r$, the number of \emph{Borel generators} of each ideal, and the degrees of Borel generators. Note that when the first parameter $r$ is equal to one, the objects are the Rees algebras of strongly stable ideals and the Koszulness of these objects are studied by controlling the remaining two parameters. Authors of \cite{BC}, \cite{D}, and \cite{dipasquale2019rees} investigate the Koszulness of the Rees algebras of strongly stable ideals (see \Cref{sec:4} for details).

In the investigation of quadric Gr\"obner bases of $\mathcal{L}$ and $\mathcal{K}$ in the multi-Rees setting, it is reasonable to start by considering classes of ideals whose Rees algebras or special fiber rings are $G$-quadratic. Emergence of this natural approach can be observed  in \cite{BC}. In one such instance, based on De Negri's result on the Koszulness of $\mathcal{F} (I)$ for a principal strongly stable ideal $I$ from \cite{D},  Bruns and Conca suggest that  it is ``very likely" that $\mathcal{F} (I_1 \oplus \cdots \oplus I_r)$ is defined by a  Gr\"obner basis of quadrics when each $I_i$ is principal strongly stable (see \cite[Page 3]{BC}). Very recently,  DiPasquale  and  Jabbar Nezhad confirm an extended version of Bruns and Conca's \emph{suggestion} for $\mathcal{R} (I_1 \oplus \cdots \oplus I_r)$ when each $I_i$ is principal strongly stable \cite[Corollary 6.4]{dipasquale2020koszul}.

In the study of Rees algebras of strongly stable ideals, it is possible to deduce the Koszulness of the Rees algebra from its special fiber thanks to a result of Herzog, Hibi and Vladoiu \cite[Theorem 5.1]{HHV}  which states that \emph{$\mathcal{R} (I)$ is of fiber type for a strongly stable ideal $I$}. The multi-Rees version of Herzog, Hibi and Vladoiu's result would be quite useful in the multi-Rees world.  We generalize their result to the multi-Rees algebra case and show that the multi-Rees algebra of strongly stable ideals is of fiber type, that is, the ideal $\mathcal{L}\subseteq S[T_1,\ldots, T_N]$ is generated by linear relations in $T_1, \ldots, T_N$ variables and the generators of $\mathcal{K}$. In particular, we prove that a Gr\"{o}bner basis of $\mathcal{L}$ is the union of a Gr\"{o}bner basis of $\mathcal{K}$ and a set of quadric binomials associated to first syzygies of $I_1 \oplus \cdots \oplus I_r$ (see \Cref{multireesisfibertype}). It is worth pointing that we allow ideals $I_1, \ldots, I_r$ to be generated in different degrees while each Borel generator of $I_i$ is of the same degree for $i \in \{1,\ldots, r\}$. An immediate application of \Cref{multireesisfibertype} is that the  multi-Rees algebra of strongly stable ideals is $G$-quadratic whenever its special fiber ring is.  Therefore,  Koszulness of the multi-Rees algebra of strongly stable ideals can be obtained from that of its special fiber ring.

In the late nineties, Conca and De Negri produce examples of strongly stable ideals indicating that the Rees algebras of strongly stable ideals with more than two Borel generators are not necessarily Koszul (see \cite[Example 1.3]{BC}). Following their footsteps, we provide a collection of examples of strongly stable ideals whose multi-Rees algebras are not Koszul  (see \Cref{sec:4}). Hence, if  $\mathcal{R} (I_1 \oplus \cdots \oplus I_r)$ is always Koszul, for a fixed arrangement of the three sets of parameters introduced earlier, (i.e. number of ideals, number of Borel generators of each ideal and degree of each Borel generator), we identify all possible values of the parameters in the arrangement (see \Cref{prop:KoszulCases}).

In \cite[Question 6.4]{dipasquale2019rees}, the authors conclude their paper with a question asking whether the multi-Rees algebra of strongly stable ideals $I_1, \ldots, I_r$ is Koszul where each $I_i$ has either one or two Borel generators. In this case, by \Cref{prop:KoszulCases}, at most two of the ideals have two Borel generators. Furthermore, if exactly two of the ideals have two Borel generators then these generators must be quadric or cubic. 
In the second half of the paper, we focus on the multi-Rees algebra of strongly stable ideals $I_1$ and $I_2$ with two quadric Borel generators. In particular, we prove in  \Cref{sec:6} that $\mathcal{F} (I_1 \oplus I_2)$ has a quadric Gr\"obner basis. We achieve it by obtaining an explicit quadric Gr\"obner basis for the toric ideal of the special fiber ring with respect to \emph{head and tail order} (see \Cref{thm:GBs7}). Hence  Koszulness of $\mathcal{R} (I_1 \oplus I_2)$ follows from \Cref{multireesisfibertype}.

Our paper is structured as follows. In \Cref{sec:prelim}, we collect the necessary terminology to be used throughout the paper. In \Cref{sec:directedGraph}, we introduce \emph{directed graph of a monomial} and use these objects to determine when a given collection of binomials (with a marked term) form a Gr\"obner basis (see \Cref{fibergraphinduces}). We prove our first main result, \Cref{multireesisfibertype}, in \Cref{sec:3}. In \Cref{sec:4}, we investigate  Koszulness of multi-Rees algebras of strongly stable ideals through examples and identify large classes of ideals whose multi-Rees algebras are not necessarily Koszul, and present our second main result, \Cref{prop:KoszulCases}. In Sections \ref{sec:5} and \ref{sec:6}, we study multi-Rees algebras of strongly stable ideals with two quadratic Borel generators. In particular, we prove our third main result, \Cref{thm:GBs7}, and its corollary, \Cref{thm:MultiReesGB}, in \Cref{sec:6}.

\subsection*{Acknowledgements}  The authors sincerely thank the reviewers for their helpful and constructive suggestions which improved the presentation of the manuscript immensely. The first named author was partially supported by the University of South Alabama Arts and Sciences Support and Development Award. Thanks to this award, part of this work was done while the second author visited the University of South Alabama. The second author thanks the 
Department of Mathematics and Statistics for their hospitality. Many of the computations related to this paper was done using Macaulay2 \cite{M2}.

\section{Preliminaries}\label{sec:prelim}
 Let $S=\KK [x_1,\ldots, x_n]$  be the polynomial ring  in $n$ variables over a field $\KK$ and suppose $S$ is equipped with the standard grading as a $\KK$-algebra. Recall that the graded reverse lexicographic order on $S$ is defined as the follows: for any $\mathbf{x}^A, \mathbf{x}^B \in S$ such that $A, B \in (\mathbb{N} \cup \{0\})^n$, we say $\mathbf{x}^A > \mathbf{x}^B$ when if either $\mathrm{degree}(\mathbf{x}^A) > \mathrm{degree}(\mathrm{x}^B)$ or  $\mathrm{degree}(\mathbf{x}^A) = \mathrm{degree}(\mathrm{x}^B)$ and the last non-zero entry of the vector of integers $A-B$ is negative.
\begin{definition}\label{DefBorelFixed}
\begin{enumerate}[(a)]
\item Let $\succ_{rlex}$ denote the graded reverse lexicographic order with the variable order $x_1>x_2> \cdots >x_n$ on $S$. We write $[n]=\{1,\ldots,n\}$. 
\item	A monomial ideal $I \subseteq S$ is called \emph{strongly stable} if for each monomial $m \in I$, we have $x_j\frac{m}{x_i} \in I$ whenever $x_i$ divides $m$ and $j <i$.
\item Fix a monomial $m \in S$, and  suppose that $x_i$ divides $m.$ Then for any $j<i$, the monomial $m'= x_j\frac{m}{x_i}$ is called a \emph{one step strongly stable reduction} of $m$.   
\item   Fix a degree $d.$ Then one can define a partial order $<$, called the \emph{strongly stable order}, on the degree $d$ monomials of $S$ by setting $m' < m$ whenever  $m'$ can be obtained from $m$ by a sequence of one step strongly stable reductions. In this case, we say $m'$ comes before $m$ in the strongly stable order.
\item  Let $\mathcal{M}=\{m_1, \ldots, m_k\}$ be a collection of degree $d$ monomials. Then the smallest strongly stable ideal containing $\mathcal{M}$ is called the \emph{strongly stable ideal generated by} $\mathcal{M}$ and denoted by $\MB(\mathcal{M})$. If $I=\MB(\mathcal{M})$, the set $\mathcal{M}$ is called a strongly stable generating set for $I$. The strongly stable ideal $I$ has a unique minimal strongly stable generating set corresponding to the latest monomial generators in the strongly stable  order and these monomials are called its \emph{strongly stable generators} or \emph{Borel generators}. For the remainder of the paper, we shall use the term Borel generators.  If $I=\MB(m)$ has only one Borel generator, we say $I$ is a prinpical strongly stable ideal. 
\end{enumerate}
\end{definition}

\begin{remark}
Strongly stable ideals are \emph{Borel-fixed} (i.e., fixed under the action of the Borel subgroup). In characteristic zero, the notion of strongly stable  coincide with Borel-fixed. 
\end{remark}

In this paper, our main objects are the multi-Rees algebras of strongly stable ideals. We recall the definition of a multi-Rees algebra and its special fiber ring below.

\begin{definition}\label{def:multi}
	Let $I_1, \dots, I_r$ be a collection of ideals in $S$. The \emph{multi-Rees algebra of} $I_1, \dots, I_r$ is the multi-graded $S$-algebra 
	\[ \mathcal{R}(I_1 \oplus \dots \oplus I_r)=\bigoplus_{(a_1, \ldots, a_r) \in \mathbb{Z}_{\geq 0}^n} I_1^{a_1} \cdots I_r^{a_r}t_1^{a_1}\cdots t_r^{a_r} \subseteq S[t_1,\ldots, t_r] \]
	where $t_1,\ldots, t_r$ are indeterminates over $S$.
	In particular, in the case $r=1$,  it is the classical Rees algebra of an ideal $I$. The multi-Rees algebra $\mathcal{R}(I_1 \oplus \dots \oplus I_r)$ is also the Rees algebra of the module $I_1\oplus \cdots \oplus I_r$. The \emph{special fiber ring of} $I_1, \dots, I_r$ is defined as $$\mathcal{F}(I_1 \oplus \dots \oplus I_r)  = \mathcal{R}(I_1 \oplus \dots \oplus I_r) \otimes_S \KK$$
where we regard $\KK$ as an $S$-algebra via $\KK \cong S/\langle x_1,\dots, x_n \rangle $.
	\end{definition}

	Let $I_i$ be a monomial ideal and  $\{u_{i,j}: 1 \leq j \leq s_i\}$ be the minimal generating set of $I_i$ for each $i \in [r]$. We can construct the presentation of $\mathcal{R}(I_1 \oplus \dots \oplus I_r)$ as follows. Consider the $S$-algebra homomorphism given by
	\begin{align*}
	 	\varphi : R= S[T_{i,j} ~:~ 1\leq i\leq r, 1\leq j \leq s_i] & \longrightarrow    S[t_1,\dots, t_r] \\
	 	T_{i,j} & \longmapsto u_{i,j} t_i
	\end{align*}
and extended algebraically. The ideal $\mathcal{L}=\ker (\varphi)$ is called the \emph{defining ideal} of the multi-Rees algebra and  the minimal generators of $\mathcal{L}$ are called the \emph{defining equations} of $\mathcal{R}(I_1 \oplus \dots \oplus I_r)$. The multi-Rees algebra of $I_1,\ldots, I_r$ is the quotient 
$$\mathcal{R}(I_1 \oplus \dots \oplus I_r) = S[T_{i,j} ~:~ 1\leq i\leq r, 1\leq j \leq s_i]  / \ker (\varphi) \cong \im \varphi.$$

The map $\varphi$ induces a surjective $\KK$-algebra homomorphism $\varphi'$ such that
	\begin{align*}
	 	\varphi ': R'=\KK [T_{i,j} ~:~ 1\leq i\leq r, 1\leq j \leq s_i] & \longrightarrow    \KK[u_{i,j}t_i ~:~ 1\leq i\leq r, 1\leq j \leq s_i] \\
	 	T_{i,j} & \longmapsto u_{i,j} t_i.
	\end{align*}
The special fiber ring, $\mathcal{F}(I_1 \oplus \dots \oplus I_r)$, is isomorphic to the image of $\varphi'$ and the ideal $\ker (\varphi ')$ is called the \emph{defining ideal} of the special fiber ring. Note that  $\mathcal{F}(I_1 \oplus \dots \oplus I_r)$ is a toric ring because $\ker (\varphi ')$ is a toric ideal associated to the monomial map $\varphi'$. We often use the notation $T(I_1 \oplus \dots \oplus I_r)$ to denote $\ker (\varphi ')$.

In what follows, we provide a quick overview of marked polynomials and Noetherian reductions in the context of Gr\"obner bases  theory.

\begin{definition}
A \emph{marked polynomial} is a polynomial $g \in S$ together with a ``marked" initial term $\text{in} (g)$, say $g =c\underline{u}+h$ where $u$ is a monomial with $c \in \KK^*$ and $h$ is a polynomial such that $u \notin \supp(h)$ and $ \text{in} (g)=u$.  Note that any of the terms appearing in $g$ can be chosen as $\text{in} (g)$.
\end{definition}

Let $\mathcal{G}= \{ g_i=c_i\underline{u_i}+h_i ~:~ 1\leq i\leq s\}$  be a collection of marked polynomials. In the following two examples, we investigate the existence of a monomial order $\tau$  such that $\mathcal{G}$  is marked coherently, namely, one has $\ini_{\tau} (g_i) = \underline{u_i}$ for each $g_i \in \mathcal{G}$.

\begin{example}
	Consider the polynomials $f=\underline{xy}-yz$ and $g=\underline{z^2}+x^2$ with the marked terms as the leading ones. If there exists a monomial order $\tau$ such that  $\ini_{\tau}  (f)= \underline{xy}$ and $\ini_{\tau} (g) = \underline{z^2}$, then  $yz <_{\tau} xy$ implies $z <_{\tau} x$, and $x^2<_{\tau}z^2$ implies that $x <_{\tau}z$ which is not possible. Hence, there is no such monomial order. In fact, for the marked initial terms, there is no monomial order that makes $\{ f,g\}$  a Gr\"obner basis for $\langle f, g \rangle$  unless $\text{char}(\mathbb{K}) \neq 2$. 
\end{example}

\begin{example}\label{2by2minors}
	Consider the ideal $I$ generated by the $2\times 2$ minors of the following matrix. 
	$$\left[ \begin{array}{ccc} x_1 & x_2 & x_3 \\ x_4 & x_5 & x_ 6 \end{array}\right]$$
	
	The ideal $I$ is generated by $f=\underline{x_1x_5}- x_2x_4$, $g=\underline{x_2x_6}-x_3x_5$ and $h=\underline{x_3x_4}-x_1x_6$. As in the previous example, there is no monomial order $\tau$ such that $\ini_{\tau} (f)= \underline{x_1x_5}$, $\ini_{\tau} (g) = \underline{x_2x_6},$ and $\ini_{\tau} (h)= \underline{x_3x_4}$. If such an order exists, then $x_2x_4 <_{\tau} x_1x_5$, $x_3x_5 <_{\tau} x_2x_6$ and $x_1x_6 <_{\tau} x_3x_4$. It implies that $(x_2x_4)(x_3x_5)(x_1x_6) <_{\tau} (x_1x_5)(x_2x_6)(x_3x_4)$ which is not possible. 
\end{example}

As we see from the above examples, given a small number of marked polynomials, it is rather straightforward to show there is no monomial order $\tau$ such that marked terms are initial terms with respect to $\tau$. If there are more marked polynomials, it is more difficult to figure out whether $\mathcal{G}$ is marked coherently. Naturally, one might ask whether there is a way to determine  when  $\mathcal{G}$ is marked coherently. We could further question when $\mathcal{G}$ is a Gr\"obner basis for the ideal generated by the collection with respect to $\tau$. These questions are answered in \cite{S} in terms of reduction relations. We conclude this section by collecting several necessary definitions and recalling the related answer from \cite{S}.

\begin{definition} 
Let $\mathcal{G}=\{g_1, \dots, g_s\}$ be a collection of marked polynomials with initial terms $\{u_1, \dots, u_s\}$ where $ \ini (g_i) = u_i$ for each $i$.  Let  $f$ be a polynomial in $S$. If there exists a monomial $u \in \supp(f)$ such that $u_i$ divides $u$ for some $1 \leq i \leq s$, let $\displaystyle f':= f - \frac{cu}{c_iu_i} g_i$ where $f = cu+h$ with $c,c_i \in \KK^*$, coefficients of $u$ and $u_i$, such that $u \notin \supp (h)$. We call the polynomial $ f'$  a \emph{one step reduction} of $f$ with respect to  $\mathcal{G}$. We denote this reduction relation by  $f \longrightarrow_{\mathcal{G}} f'$.
\end{definition}

 A collection of marked polynomials $\mathcal{G}=\{g_1, \dots, g_s\}$ is said to define a \emph{Noetherian reduction relation} if the number of successive one step reductions of any polynomial is finite, i.e., for a polynomial $f$ and a chain of reduction relations $f \longrightarrow_{\mathcal{G}} f' \longrightarrow_{\mathcal{G}}  \cdots \longrightarrow_{\mathcal{G}} f''$, there is no possible one step reduction of $f''$ with respect to $\mathcal{G}$.

\begin{theorem}\cite[Theorem 3.12]{S}\label{reductionrelationNoetherian}
	Given a collection of marked polynomials $\mathcal{G}=\{g_1, \dots, g_s\}$, there exists a monomial order, $\tau$, such that $\ini_{\tau} (g_i)=u_i$ and $\mathcal{G}$ is a Gr\"obner basis for $\langle g_1, \dots, g_s\rangle$ with respect to $\tau$ if and only if the reduction relation $\longrightarrow_{\mathcal{G}}$ is Noetherian.
\end{theorem}

\section{Directed graphs of monomials via reductions}\label{sec:directedGraph}

In this section, we consider collections of marked polynomials consisting exclusively of homogeneous binomials, i.e., $\mathcal{G}= \{g_1,\dots, g_s\}$ such that $g_i=\underline{u_{i,1}}-u_{i,2}$  where $u_{i,1}$ and $ u_{i,2}$ are both monomials for all $i$. We first introduce a new notion called the \emph{directed graph} of a monomial in $S$ and use its combinatorial structure to detect when $\mathcal{G}$ is a Gr\"obner basis with respect to some monomial order in \Cref{fibergraphinduces}.

\begin{definition}
	Let $\mu$ be a monomial in $S$ and $\mathcal{G}= \{g_1=\underline{u_{1,1}}-u_{1,2},\dots, g_s=\underline{u_{s,1}}-u_{s,2}\}$ be a collection of marked binomials.  The \emph{directed graph} of $\mu$ with respect to $\mathcal{G},$ denoted by $\Gamma_{\mu} (\mathcal{G})$, is defined as follows:
	
	\begin{itemize}
	    \item The vertices are the monomials that can be obtained from $\mu$  by a sequence of one step reductions with respect to $\mathcal{G}$. 
	    \item For two vertices $v$ and $w$, there is a directed edge from $v$ to $w$  if $v \longrightarrow_{\mathcal{G}} w$, i.e., $w$ is a one step reduction of $v$ with respect to $\mathcal{G}$. 
	\end{itemize}

\end{definition}

\begin{example}\label{incompleteGrobner}
	Consider the collection of marked polynomials $\mathcal{G}=\{\underline{x_1x_3}-x_2^2,\underline{x_1x_2}-x_3^2\}$ and the monomial $x_1x_2x_3$. Let $g_1 =\underline{x_1x_3}-x_2^2$ and $g_2= \underline{x_1x_2}-x_3^2$. Then the directed graph of $x_1x_2x_3$ with respect to $\mathcal{G}$ is given below. Notice that this directed graph has two different sinks.

	
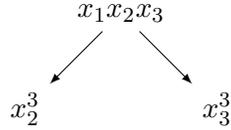
\begin{figure}[hbt]
    \begin{center}
        \begin{tikzpicture}[scale=0.85,->,>=stealth]
         \node (0) at (0,0) {$x_1x_2x_3$};
         \node (1) at (1.5,-1.5) {$x_3^3$};
        \node (2) at (-1.5,-1.5) {$x_2^3$};
            \draw[-latex]  (0)--(1);
             \draw[-latex]  (0)--(2);
        \end{tikzpicture}
    \end{center}
    \end{figure}

One can see that the marked monomials are the leading terms for the collection $\mathcal{G}$ with respect to the (degree) lexicographic order where $x_1 > x_2 > x_3$. However,  $\mathcal{G}$ is not a Gr\"obner basis because the $S$-polynomial of $g_1$ and $g_2$ given below
$$S(g_1,g_2)= x_2g_1-x_3g_2=x_3^3-x_2^3$$
has a non-zero remainder when divided by $\mathcal{G}$.

\end{example}

\begin{example}\label{trueGrobner}
	Consider  $\mathcal{G}=\{\underline{x_1x_4}-x_2x_5,\underline{x_2x_3}-x_4^2\}$ and  the monomial $x_1x_2x_3x_4$. The directed graph of $x_1x_2x_3x_4$ with respect to $\mathcal{G}$ is given as follows.

	
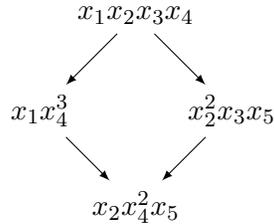
\begin{center}
		\begin{tikzpicture}[scale=0.85]
		\begin{scope}
		\node (1) at (3, 3) {$x_1x_2x_3x_4$};
		\node (2) at (1.5, 1.5) {$x_1x_4^3$};
		\node (3) at (4.5, 1.5) {$x_2^2x_3x_5$};
		\node (4) at (3, 0) {$x_2x_4^2x_5$};
		\end{scope}
		\draw[-latex] (1) edge [above]  (2);
		\draw[-latex] (1) edge [right] (3);
	    \draw[-latex] (2) edge [above]  (4);
        \draw[-latex] (3) edge [right] (4);		
		\end{tikzpicture}
	\end{center}
	
Note that the directed graph possesses a unique sink and has no (directed) cycles. Additionally, observe that $\mathcal{G}$ is a Gr\"obner basis with respect to the monomial order $\succ_{rlex}$.
\end{example}

\begin{example}\label{cycles}
	Consider the monomial $x_1x_3x_5x_6$ and  $\mathcal{G}=\{\underline{x_1x_5}-x_2x_4,\underline{x_2x_6}-x_3x_5,\underline{x_3x_4}-x_1x_6\}$ from Example \ref{2by2minors}. Below is the directed graph of $x_1x_3x_5x_6$ with respect to $\mathcal{G}$.
	
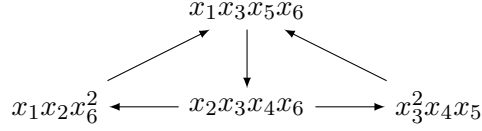
	\begin{center}
		\begin{tikzpicture}[scale=0.85]
		
		\begin{scope}
		\node (1) at (1.5, 1.5) {$x_1x_3x_5x_6$};
		\node (2) at (1.5, 0) {$x_2x_3x_4x_6$};
		\node (3) at (4.5, 0) {$x_3^2x_4x_5$};
		\node (4) at (-1.5, 0) {$x_1x_2x_6^2$};
		\end{scope}
		\draw[-latex] (1) edge [above]  (2);
		\draw[-latex] (2) edge [right] (3);
		\draw[-latex] (2) edge [above]  (4);
		\draw[-latex] (3) edge [right] (1);		
		\draw[-latex] (4) edge [right] (1);	
		\end{tikzpicture}
	\end{center}
It is clear that the directed graph has no sinks and it has two cycles. Additionally, recall from \Cref{2by2minors} that $\mathcal{G}$ is not a Gr\"obner basis. 
\end{example}

In the light of the previous examples, one can ask whether $\Gamma_{\mu} (\mathcal{G})$ possesses a unique sink and has no cycles if there exists a monomial order $\tau$ such that the collection of marked binomials $\mathcal{G}$  is a Gr\"obner basis with respect to $\tau$ and the initial terms are the marked monomials. In the next theorem, we answer this question and further show that the converse is true.

\begin{theorem}\label{fibergraphinduces}
Let $\mathcal{G}= \{g_1=\underline{u_{1,1}}-u_{1,2},\dots, g_s=\underline{u_{s,1}}-u_{s,2}\}$ be a collection of marked binomials. For every monomial $\mu$ in $S$, the directed graph of $\mu$ with respect to $\mathcal{G}$ possesses a unique sink and has no directed cycles if and only if there exists a monomial order $\tau$ such that $\mathcal{G}$ is a Gr\"obner basis with the marked monomials as the initial terms.
\end{theorem}

\begin{proof}

If there exists a monomial $\mu \in S$ such that $\Gamma_{\mu} (\mathcal{G})$ has a cycle, then there is an infinite sequence of one step reductions with respect to $\mathcal{G}$ starting at $\mu$. Thus $\mathcal{G}$  fails to define a Noetherian reduction relation. It follows from  \Cref{reductionrelationNoetherian} that there is no monomial order $\tau$ such that $\mathcal{G}$ is a Gr\"obner basis with the marked monomials as the initial terms.

If there exists a monomial $\mu \in S$ such that $\Gamma_{\mu} (\mathcal{G})$ has more than one sink, then division of $\mu$ by $\mathcal{G}$ does not have a unique remainder since each sink is the remainder of $\mu$ when it is divided by $\mathcal{G}$. Thus,  $\mathcal{G}$ can not be a Gr\"obner basis with respect to a monomial order $\tau$ such that the marked monomials as the initial terms with respect to $\tau$. Therefore, it is necessary that  the directed graph has a unique sink and no cycles for the existence of a monomial order $\tau$ such that $\mathcal{G}$  is a Gr\"obner basis with leading terms as the marked monomials.

Suppose that $\Gamma_{\mu} (\mathcal{G})$ has a unique sink and no cycles for every monomial $\mu$ in $S$. Our goal is to prove that there exists a monomial order $\tau$ such that $\mathcal{G}$ is a Gr\"obner basis with the marked monomials as the initial terms. It suffices to show that $\mathcal{G}$ defines a Noetherian reduction relation by \Cref{reductionrelationNoetherian}.  For this purpose, we define an invariant called the longest path length, denoted by $\ell_{\max}(\mu)$, where \[\ell_{\max}(\mu)=\{\text{the longest distance from } \mu \text{ to the unique sink in } \Gamma_{\mu} (\mathcal{G}) \}.\]
 The longest path length $\ell_{\max}(\mu)$ is well-defined because there can be only a finite number of directed paths starting at $\mu$ and ending at the unique sink  in  $\Gamma_{\mu} (\mathcal{G})$.  Additionally,  the longest path length is a nonnegative integer, i.e., $\ell_{\max}(\mu)\geq 0$. We define a related invariant for a polynomial $f$ in $S$ and it is given by  
 $$\displaystyle \ell_{\max}(f):=\sum_{\mu \in \text{supp} (f) } \ell_{\max}(\mu).$$
If $f$ has no one step reductions with respect to $\mathcal{G}$, we can choose another polynomial in $S$ until that polynomial has a one step reduction. Without loss of generality, we may assume that there exists a polynomial $f'$ such that $f \longrightarrow_{\mathcal{G}} f'$.  Then there exists a monomial $v \in \supp(f)$  such that $u_{i,1}$ divides $v$ for some $1 \leq i \leq s$ and
$$f'= f - \frac{c_v v}{u_{i,1}}g_i = (f-c_vv)+c_v \frac{v}{u_{i,1}}u_{i,2}$$ 
  where $f = c_vv+h$ with $c_v \in \KK^*$ such that $v \notin \supp (h)$.
 
In order to prove $\mathcal{G}$ defines a Noetherian reduction relation, it suffices to show $\ell_{\max}(f')< \ell_{\max}(f)$.  Let $ \displaystyle v'= \frac{v}{u_{i,1}}u_{i,2}$. If  $ v' \in \supp (f)$, then
 $$\ell_{\max}(f')\leq \left(\sum_{\mu \in \supp (f)} \ell_{\max}(\mu) \right) -\ell_{\max} (v) <\sum_{\mu \in \supp (f)} \ell_{\max}(\mu)=\ell_{\max}(f).$$
The last inequality is due to the fact that $\ell_{\max}(v)>0$ because $v$ can not be the unique sink in $ \Gamma_v (\mathcal{G})$ as $ \displaystyle  v  \longrightarrow_{\mathcal{G}} v' $. If $v' \notin \supp (f)$, then 
	$$\ell_{\max}(f')=\displaystyle \left(\sum_{\mu \in \supp (f)} \ell_{\max}(\mu) \right) -\ell_{\max} (v)+\ell_{\max} (v')<\sum_{\mu \in \supp (f)} \ell_{\max}(\mu)=\ell_{\max}(f).$$
The strict inequality follows from the fact $ \ell_{\max}(v)>\ell_{\max} (v') $ which is proved as follows.  Since $   v  \longrightarrow_{\mathcal{G}} v' $, the directed graph of  $ v'$ with respect to $\mathcal{G}$ is a subgraph of $ \Gamma_v (\mathcal{G})$; moreover,   both directed graphs share the same sink. Thus, the existence of a directed edge from $v$ to $v'$ guarantees that  $ \ell_{\max}(v)>\ell_{\max} (v') $. 
\end{proof}
 
 \begin{remark} Let $\mathcal{G}$ be a Gr\"obner basis consisting exclusively of binomials with leading terms as the marked monomials. Then for any polynomial $f \in S$, there is always a unique $\tilde{f}$ obtained through a sequence of one step reductions with respect to $\mathcal{G}$ such that $\ell_{max}(\tilde{f})=0$.  It is also worth noting that  $\tilde{f}$ is  the remainder of $f$ when it is divided with respect to $\mathcal{G}$. 
 \end{remark}

\section{Strongly Stable Ideals and their multi-Rees Algebras}\label{sec:3}

In the paper, \cite{HHV}, Herzog, Hibi and Vladoiu determine the defining equations of the Rees algebra of a strongly stable ideal as the union of a Gr\"obner basis for the toric ideal of its special fiber and a set of first syzygies, proving that Rees algebras of strongly stable ideals are of fiber type. The main result of this section (\Cref{multireesisfibertype}) generalizes their result to the case of multi-Rees algebras.

\begin{notation}
	Let $I_1, \dots, I_r$ be a collection of strongly stable ideals in $S$ and $\{u_{i,1}, \dots,u_{i,s_i}\}$ be  the set of minimal monomial generators of $I_i$ for each $i$ such that $\deg (u_{i,j})=d_i$ for all $1 \leq j \leq s_i$ and 
 	$$u_{i,1} \succ_{rlex} u_{i,2} \succ_{rlex} \dots \succ_{rlex} u_{i,s_i}$$ with respect to reverse lexicographic order defined in \Cref{DefBorelFixed}.
 	Recall that the defining ideals of the multi-Rees algebra $\mathcal{R}(I_1 \oplus \dots \oplus I_r)$ and its special fiber  $\mathcal{F}(I_1 \oplus \dots \oplus I_r)$ are denoted  by $\mathcal{L} \subseteq R$ and $T(I_1 \oplus \dots \oplus I_r) \subseteq R'$, respectively, where   
 	$$R=S[T_{i,j} : 1\leq i\leq r, ~1\leq j \leq s_i] \text { and } R'= \KK[T_{i,j} : 1\leq i\leq r, ~1\leq j \leq s_i].$$
\end{notation}
In what follows, we present the main result of this section and establish its proof through two auxiliary lemmas. Explicit statements and proofs of these lemmas are provided in the following steps. 

\begin{theorem}\label{multireesisfibertype}
 	If $\mathcal{G}'$ is a Gr\"obner basis for the toric ideal $T(I_1 \oplus \dots \oplus I_r)$ with respect to a monomial order $\tau'$ on $R'$, then 
 	\[ \mathcal{G} :=  \left\{ \underline{x_iT_{l,k}}-x_jT_{l,k'} ~:~ i<j  \text{ and } x_iu_{l,k}=x_ju_{l,k'}  \right\} \cup \mathcal{G}' \] is a Gr\"obner basis for  the defining ideal of the multi-Rees algebra $\mathcal{R}(I_1 \oplus \cdots \oplus I_r)$ with respect to an extended monomial order $\tau$ on $R$ such that $\tau |_{R'}=\tau'$.
 \end{theorem}

\begin{proof}
In \Cref{lem:aux1}, we show that the collection of binomials  $\mathcal{G}$ defines a Noetherian reduction relation. It then follows from \Cref{reductionrelationNoetherian} that there exists a monomial term order $\tau $ such that $\mathcal{G}$ is a Gr\"obner basis for the ideal generated by $\mathcal{G}$ with respect to $\tau$. Finally, in \Cref{lem:aux2}, we prove that the initial ideal of $\langle   \mathcal{G} \rangle $ and the initial ideal of $\mathcal{L} $ with respect to $\tau$ are equal. Thus, the collection $\mathcal{G}$ is a Gr\"obner basis for  $\mathcal{L}$.
\end{proof}

  \begin{lemma}\label{lem:aux1}
 Let $\mathcal{G}'$ be a Gr\"obner basis for the toric ideal  $T(I_1 \oplus \dots \oplus I_r)$ with respect to a monomial order $\tau'$ in $R'$.  The collection of marked binomials
 $$ \mathcal{G} = \left\{ \underline{x_iT_{l,k}}-x_jT_{l,k'}~:~ i<j  \text{ and } x_iu_{l,k}=x_ju_{l,k'}  \right\} \cup \mathcal{G}'$$
 defines a Noetherian reduction relation in $R$.
 \end{lemma}

 \begin{proof} 
Let $u$  be a monomial in $R'$ which is presented by
$$u = \prod_{i=1}^r \prod_{j=1}^{c_i} T_{i,k_{i,j}}= \Big( T_{1,k_{1,1}} \cdots T_{1,k_{1,{c_1}}} \Big)  \cdots \Big( T_{r,k_{r,1}} \cdots T_{r,k_{r,{c_r}}} \Big)$$
 	where $1 \leq k_{i,1} \leq \cdots \leq  k_{i,{c_i}} \leq  s_i$	for all $1\leq i\leq r$.
We define the \emph{content} of $u$ as the monomial in $S$, denoted by $ \textbf{x}^{\alpha_u} $, and  given by
\[\textbf{x}^{\alpha_u}  := x_1^{\alpha_{u,1}} \cdots x_n^{\alpha_{u,n}} = \prod_{i=1}^r \prod_{j=1}^{c_i} u_{i,k_{i,j}}.\]
Let  $m$ be a monomial in $S$  expressed as $ \displaystyle m= \prod_{j=1}^{c}  x_{i_j} $ where $i_j \in \{1, \ldots, n\}$.  

 In the remainder of the proof, we use the following order for the elements of $\NN^2.$ Given  $\mathbf{a}, \mathbf{b} \in \mathbb{N}^2$, we say  $\mathbf{a} <_{lex,2} \mathbf{ b}$ if the first nonzero entry of $\mathbf{b-a}$ is positive.

In order to show  that $\mathcal{G}$ defines a  Noetherian reduction relation in $R$, we first prove that  it is true for all monomials in $R$. For this purpose, we introduce an invariant for a monomial  $mu$ in $R$, where $m$ and $u$ are given as above, by associating it to an ordered pair $(o_{mu},l_{mu})$ where 
$$\displaystyle o_{mu} = \sum_{q=1}^{c} \left( \sum_{t=i_q+1}^{n} \alpha_{u,t}\right) \text{ and } l_{mu}=\ell_{\max}(u)$$
such that $l_{mu}$ is the longest path length invariant defined in \Cref{fibergraphinduces} and $o_{mu}$ can be interpreted as the iterated cumulative degree of the content of $u$ with respect to $m$.  If $i_q=n$, then the sum is empty. Notice that the order of the factorization of $m$ does not make a difference for defining $o_{mu}$. Thus,  it suffices to show  the following claim to prove $\longrightarrow_{\mathcal{G}}$ is Noetherian for monomials in $R$.
$$\text{ If }  mu \longrightarrow_{\mathcal{G}}  m'u' \text{ then } (o_{m'u'},l_{m'u'}) <_{lex,2} (o_{mu},l_{mu}).$$  
If there exists a reduction relation $mu \longrightarrow_{\mathcal{G}}  m'u'$, it must be associated to an element of $\mathcal{G}'$ or $\mathcal{G} \setminus \mathcal{G}'$. We call this element a reduction element. If the reduction element is in $\mathcal{G}',$  it follows from \Cref{fibergraphinduces} that $\ell_{\max}(u')<\ell_{\max}(u)$. Since $m=m'$ and the contents of $u$ and $u'$ are the same, it is immediate that $o_{mu} = o_{m'u'}$. Thus, $(o_{m'u'},l_{m'u'}) <_{lex,2} (o_{mu},l_{mu})$, as desired.

 If the reduction element is in $\mathcal{G} \setminus \mathcal{G}'$, it is of the form $\underline{x_i T_{l,k}} - x_j T_{l,k'}$ where $i <j$ and $x_i u_{l,k} = x_j u_{l,k'}$. Then, there exists a variable in the support of $m$ such that  $x_{i_p} = x_i$ for some $p \in \{1, \ldots, c\}$ and
 $$\displaystyle m' =x_j \frac{m}{x_i}= \prod_{q=1}^{c}  x_{i'_q} \text{ and } u'=T_{l,k'}\frac{u}{T_{l,k}}$$
 where $x_{i'_p} = x_j$ and $x_{i'_q} = x_{i_q}$ for all the remaining $q \in \{  1, \ldots, c\}.$ Since $x_i u_{l,k} = x_j u_{l,k'}$, the contents of $u$ and $u'$ only differ at the variables $x_i$ and $x_j$. In particular, $ \alpha_{u,t} = \alpha_{u',t}$ for all $t \neq i,j$ while  $ \alpha_{u,i} = \alpha_{u',i}-1$ and $ \alpha_{u,j} = \alpha_{u',j} +1.$ By making use of these equalities and recalling that $i_p=i$, we obtain $o_{mu}> o_{m'u'}$ through the following term by term comparison.
 \begin{align*}
   o_{mu}   & =  \sum_{\substack{t=i+1}}^{n} \alpha_{u,t}+
     \sum_{\substack{q=1 \\ q \neq p \text{ and } i_q < i }}^{c} \left( \sum_{t=i_q+1}^{n} \alpha_{u,t}\right)+
      \sum_{\substack{q=1 \\ q \neq p \text{ and } i_q \geq i }}^{c} \left( \sum_{t=i_q+1}^{n} \alpha_{u,t}\right)\\
    & >  \sum_{\substack{t=j+1}}^{n} \alpha_{u,t}+
     \sum_{\substack{q=1 \\ q \neq p \text{ and } i_q < i }}^{c} \left( \sum_{t=i_q+1}^{n} \alpha_{u,t}\right)+
      \sum_{\substack{q=1 \\ q \neq p \text{ and } i_q \geq i }}^{c} \left( \sum_{t=i_q+1}^{n} \alpha_{u,t}\right).
 \end{align*}
On the other hand, we have 
\[   \sum_{\substack{q=1 \\ q \neq p \text{ and } i_q < i }}^{c} \left( \sum_{t=i_q+1}^{n} \alpha_{u,t}\right)=\sum_{\substack{q=1 \\ q \neq p \text{ and } i_q < i }}^{c} \left( \sum_{t=i'_q+1}^{n} \alpha_{u',t}\right).\] Also if $i_q+1>j$, then $\sum_{t=i_q+1}^{n} \alpha_{u,t}=\sum_{t=i_q'+1}^{n} \alpha_{u',t}$, and, if $i_q+1\leq j$, then $\sum_{t=i_q+1}^{n} \alpha_{u,t}=1+\sum_{t=i_q'+1}^{n} \alpha_{u',t}$. Then we get 
\[
\sum_{\substack{q=1 \\ q \neq p \text{ and } i_q \geq i }}^{c} \left( \sum_{t=i_q+1}^{n} \alpha_{u,t}\right)=k+\sum_{\substack{q=1 \\ q \neq p \text{ and } i_q \geq i }}^{c} \left( \sum_{t=i_q'+1}^{n} \alpha_{u',t}\right)
\text{ for some }k\geq 0.\]
Therefore, we obtain 
\begin{align*}
   o_{mu}   & >  \sum_{\substack{t=j+1}}^{n} \alpha_{u,t}+
     \sum_{\substack{q=1 \\ q \neq p \text{ and } i_q < i }}^{c} \left( \sum_{t=i_q'+1}^{n} \alpha_{u',t}\right)+
      \sum_{\substack{q=1 \\ q \neq p \text{ and } i_q \geq i }}^{c} \left( \sum_{t=i_q'+1}^{n} \alpha_{u',t}\right)+k \geq o_{m',u'}.
 \end{align*}
 
Hence, $(o_{m'u'},l_{m'u'}) <_{lex,2} (o_{mu},l_{mu})$ which implies that $\mathcal{G}'$ defines a  Noetherian reduction relation for all monomials in $R$.

As in the proof of \Cref{fibergraphinduces}, for a polynomial $f$ in $R,$ one can define an ordered pair  $(o_f, l_f)$ as the sum of all the ordered pairs of the monomials in the support of $f$. By repeating the arguments employed in the proof of \Cref{fibergraphinduces}, it can be shown that $ \mathcal{G}$ defines a  Noetherian reduction relation in $R$.
 \end{proof}

  \begin{lemma}\label{lem:aux2}
  Let $\mathcal{G}$ be the collection given in \Cref{lem:aux1}. Then there exists a term order $\tau$ such that $\langle \ini_{\tau } ( \mathcal{G}) \rangle = \langle \ini_{\tau }  (\mathcal{L}) \rangle$. 	
 \end{lemma}
\begin{proof} 
Recall from  \Cref{lem:aux1} that  $\mathcal{G}$ defines a Noetherian reduction relation. Then, by \Cref{reductionrelationNoetherian}, there exists a term order  $\tau$ such that $\mathcal{G}$ is a Gr\"obner bases with respect to  $\tau$ where the leading terms are the marked monomials of $\mathcal{G}$.  It is clear that $\langle \mathcal{G} \rangle \subseteq \mathcal{L}$. Our goal is to show 
    $  \langle \text{in}_{\tau }  (\mathcal{L} ) \rangle \subseteq \langle \text{in}_{\tau } ( \mathcal{G} ) \rangle$.

Let $u, m, o_{mu}$ and $\ell_{mu}$  be given as in the proof of \Cref{lem:aux1}. In what follows, we observe that if $mu$ does not belong to $\langle \text{in}_{\tau } (\mathcal{G}) \rangle $, then $o_{mu}=0=l_{mu}$. Suppose  $o_{mu}>0$. It follows from the definition of $o_{mu}$ that there exists a variable $x_i$ dividing $m$ such that $\alpha_{u,j}>0$ for some $j \geq i+1$; moreover, $x_j$ divides $u_{l,k} $ for some $l,k$ and $T_{lk}$ divides $u$ where $l\in [r] $ and $ k \in \{1, \ldots, s_l\}$. Since $u_{l,k} \in I_l$ and $I_l$ is a strongly stable ideal, there exists $k'>k$ such that $ x_i u_{l,k}=x_j u_{l,k'}$. It implies that $\underline{x_iT_{l,k}} - x_jT_{l,k'} \in \mathcal{G}'$ while $x_iT_{lk}$ divides $mu$. Thus $mu \in \langle \text{in}_{\tau } ( \mathcal{G}) \rangle$.

If $l_{mu} >0,$ then $u$ can not be the sink of the directed graph $\Gamma_u(\mathcal{G})$ which implies that there is an element $g=\underline{g_1}-g_2$ in $ \mathcal{G}$ such that $g_1= \ini_{\tau} (g)$ divides $u$. Hence, $mu$ is divisible by $\ini_{\tau} (g)$, proving that $mu \in \langle \text{in}_{\tau } ( \mathcal{G} ) \rangle$. This completes the proof of the observation.

Proceed by contradiction and assume that there is a binomial generator  of $\mathcal{L}$, say $h$, such that $\ini_{\tau} (h) \notin \langle \text{in}_{\tau } ( \mathcal{G} ) \rangle$. Let $h'$ be the unique remainder of $h$ when it is divided by $\mathcal{G}$ and let $h'=mu -m'u'$ where
$$m= \prod_{q=1}^{c_0} x_{i_q}, ~~ m'= \prod_{q=1}^{b_0} x_{j_q},~~ u =\prod_{i=1}^r \prod_{j=1}^{c_i} T_{i,k_j}, ~~u' =\prod_{i=1}^r \prod_{j=1}^{b_i} T_{i,k'_j}$$
such that the monomials $mu, m'u'$ do not belong to $\langle \text{in}_{\tau } ( \mathcal{G}) \rangle$. Since $h' \in \mathcal{L}=\ker (\varphi),$ we have the following equality which in turn implies  that $b_i=c_i$ for all $i \in [r]$.
\begin{align}\label{eq:bigprod}
    \left( \prod_{q=1}^{c_0} x_{i_q} \right) \underbrace{\left( x_1^{\alpha_{u,1}} \cdots x_n^{\alpha_{u,n}} \right)}_{\text{content of } u} \left(t_1^{c_1} \cdots t_r^{c_r} \right)  & = \left( \prod_{q=1}^{b_0} x_{j_q} \right) \underbrace{\left( x_1^{\alpha_{u',1}} \cdots x_n^{\alpha_{u',n}}\right)}_{\text{content of } u'}  \left( t_1^{b_1} \cdots t_r^{b_r}\right) 
 \end{align}

 Since $\deg (u_{i,j})=d_i$ for all $j \in \{1, \ldots, s_i\}$, one can conclude that total degrees of contents of both $u$ and $u'$ are equal; moreover, $$\sum_{i=1}^n \alpha_{u,i} = \sum_{i=1}^n \alpha_{u',i} = \sum_{i=1}^r c_i d_i,$$
implying $b_0=c_0$. Let $c= \max \{ k ~:~ \alpha_{u,k} \neq 0\}$ and $d= \max \{ k ~:~ \alpha_{u',k} \neq 0\}$. It follows from the observation that $o_{mu}=0=o_{m'u'}$. Then one can obtain the refined expressions of $m$ and $m'$ given below.
$$m= \prod_{i=c}^n x_i^{\gamma_i}, ~~~~~ m' = \prod_{i=d}^n x_i^{\delta_i} $$
As a result, \Cref{eq:bigprod} can be rewritten as
\begin{equation}\label{eq:simplified}
  \prod_{i=1}^cx_i^{\alpha_{u,i}} \prod_{i=c}^n x_i^{\gamma_i} = \prod_{i=1}^d x_i^{\alpha_{u',i}} \prod_{i=d}^n x_i^{\delta_i}  
\end{equation}

We may assume  $c\leq d$. Then \Cref{eq:simplified} implies that  $\alpha_{u,i}=\alpha_{u',i}$  when $i<c$ and $\gamma_i = \delta_i$ when $d<i$. Note that $\displaystyle \sum_{c\leq i\leq d} \gamma_i =\delta_d$ since $\deg (m)=\deg (m')$ and $\gamma_i = \delta_i$ when $d<i$. 

One can further show that $c=d$. Otherwise,  \Cref{eq:simplified} results with  the following set of equalities.
$$ \gamma_c  = \alpha_{u',c}- \alpha_{u,c},  ~~  \gamma_i = \alpha_{u',i} \text{ for } c< i<d,
~~    \gamma_d = \alpha_{u',d}+\delta_d$$
Subtituting the above equalities in $\displaystyle \sum_{c\leq i\leq d} \gamma_i =\delta_d$ yields to 
$$\alpha_{u,c} -  \alpha_{u',c}= \sum_{c< i \leq d}  \alpha_{u',i} = - \gamma_c \leq 0,$$
a contradiction because $\alpha_{u',d} >0$ by definition.  Since $c=d$, one must have $\gamma_c=\delta_c$. In addition, we also have $\alpha_{u,c} =\alpha_{u',c} $ from \Cref{eq:simplified}. Thus $\gamma_i=\delta_i$ and $\alpha_{u,i}=\alpha_{u',i}$ for all $1\leq i\leq c$. 

As a result, we conclude $m=m'$ and contents of $u$ and $u'$ are equal.  Hence, $h' = m(u-u')$ where $u-u' \in \ker (\varphi ')= \langle \mathcal{G} \rangle $. Then $\ini_{\tau} (h')= m\overline{u}$ where $\overline{u} $ is in the ideal generated by the leading terms of $\mathcal{G}$ which is a subset of $\langle \ini_{\tau } ( \mathcal{G}) \rangle $. Thus    $\ini_{\tau} (h')= m\overline{u} \in \langle \ini_{\tau } (\mathcal{G}) \rangle,$ a contradiction because none of the terms of $h'$ belongs to $\langle \ini_{\tau } (\mathcal{G}) \rangle$. Therefore, $\langle \ini_{\tau } ( \mathcal{G}) \rangle = \langle \ini_{\tau }  (\mathcal{L}) \rangle$.
\end{proof}

\section{Koszulness of the Multi-Rees Algebras of Strongly Stable Ideals}\label{sec:4}

In this section, we provide a collection of examples of strongly stable ideals whose multi-Rees algebras are not necessarily Koszul. One can systematically study the multi-Rees algebras of strongly stable ideals $I_1, \ldots, I_r$  by considering the following parameters.
\begin{itemize}
    \item $r:$ the number of ideals,
    \item $g_i:$   the number of Borel generators of $I_i$ where $I_i = \mathcal{B} (m_{i,1}, \ldots, m_{i,g_i})$,
    \item $d_i:$ the degree of Borel generators of $I_i$ where $\deg (m_{i,j}) = d_i$ for $1\leq j \leq g_i$.
\end{itemize}
 
Our examples allow us to identify the possible sets of values for the above parameters, if the multi-Rees algebra of any collection of strongly stable ideals with these parameters is always Koszul.

 If the first parameter value is set equal to one ($r=1$), we are in the world of  Rees algebras of strongly stable ideals. A closer look into the literature on these objects support our systematic approach. In particular, in \cite{D}, De Negri proves that the special fiber ring $\mathcal{F}(I)$ of the Rees algebra $\mathcal{R}(I)$ is Koszul when $I$ is a principal strongly stable ideal ($g_1=1$). In this case, one can further conclude that $\mathcal{R}(I)$ is Koszul thanks to \cite[Theorem 5.1]{HHV}, a result of Herzog, Hibi and Vladoiu stating that the Rees algebras of strongly stable ideals are of fiber type. In \cite[Example 1.3]{BC}, Bruns and Conca present an example of a strongly stable ideal with three Borel generators ($g_1=3$) such that the toric ideal has a minimal cubic generator. Thus, the special fiber ring or the Rees algebra are not necessarily Koszul when $g_1\ge 3$. Motivated by this example, authors of \cite{dipasquale2019rees} studies the Rees algebra of a strongly stable ideal $I$ with two Borel generators ($g_1=2$). Furthermore, they complete the general study on Koszulness of the Rees algebras of strongly stable ideals by proving  $\mathcal{R} (I)$ is always Koszul when $g_1=2$. 
 
 In the multi-Rees setting, it is natural to consider the classes of strongly stable ideals whose Rees algebras are always Koszul, in other words, $g_i\in \{1,2\}$ for each $1\leq i\leq r$. In fact,  authors of \cite{dipasquale2019rees}  conclude their paper with the question asking whether the multi-Rees algebras of strongly stable ideals $ I_1, \ldots, I_r$ is Koszul when $1\leq g_i\leq 2$ for each $i$. In a recent paper \cite{dipasquale2020koszul}, it is proved that the multi-Rees algebra of principal strongly stable ideals is Koszul, answering the question for the case $g_i=1$ for each $i.$  Note that there are no restrictions on the last set of parameters $d_1, \ldots, d_r$.

In what follows, we provide several examples to illustrate some conditions on $r, g_i, d_i$ that diminish the possibility of a multi-Rees algebra of  strongly stable ideals to be Koszul. As a natural consequence of these examples, one can collect the possible conditions on $r,g_i,d_i$ allowing Koszulness.

\begin{example}\label{boundnumberideals}
 Consider the following strongly stable ideals with two Borel generators
    $$I_1 = \MB (x_3^2x_6^a,x_1x_5x_6^a), ~~I_2= \MB(x_3^2x_6^b,x_2x_4x_6^b),~~ I_3=\MB(x_2x_4x_6^c,x_1x_5x_6^c),$$
    where $a,b,c$ non-negative integers. Then the  syzygy   \[(x_1x_5x_6^at_1)(x_3^2x_6^bt_2)(x_2x_4x_6^ct_3)=(x_3^2x_6^at_1) (x_2x_4x_6^bt_2) (x_1x_5x_6^ct_3)\] 
    is minimal and it corresponds to a cubic minimal generator of the  toric ideal $T(I_1\oplus I_2 \oplus I_3)$.
    
This example shows that the multi-Rees algebra of three or more strongly stable ideals (none of which are principal) is not necessarily Koszul ($r\geq 3$, $g_i \geq 2$, and $d_i\geq 2$ for each $1\leq i \leq r$). 

\end{example}

\begin{example}\label{boundondegree}
Consider the following strongly stable ideals with two  Borel generators 
    $$I_1 = \MB (x_1^2x_3^2x_4^a,x_1x_2^2x_3x_4^a), ~~I_2= \MB(x_1^2x_3^2x_4^b,x_2^4x_4^b)$$
    where $a,b$ non-negative integers. Then the syzygy
  \[(x_1^2x_3^2x_4^at_1)^2(x_2^4x_4^bt_2)=(x_1x_2^2x_3x_4^at_1)^2(x_1^2x_3^2x_4^bt_2)\]   
 is minimal and it corresponds to a cubic minimal generator of the  toric ideal $T(I_1\oplus I_2)$.

This example shows that the multi-Rees algebra of two strongly stable ideals with two Borel generators of degree four or higher is not necessarily Koszul ($r=2$, $g_1=g_2= 2,$ and $d_1,d_2 \geq 4$). 
\end{example}

\begin{example}\label{secondboundondegree}
Consider the following strongly stable ideals with two  Borel generators
    $$I_1 = \MB  (x_1x_3,x_2^2),~~I_2= \MB (x_1^2x_3^2x_4^a,x_2^4x_4^a)$$
where $a$ is a non-negative integer. Then the syzygy  $$(x_1x_3t_1)^2(x_2^4x_4^at_2)=(x_2^2t_1)^2(x_1^2x_3^2x_4^at_2)$$ 
 is minimal and it corresponds to a cubic minimal generator of the  toric ideal $T(I_1\oplus I_2)$.

This example shows that the multi-Rees algebra of two strongly stable ideals with two Borel generators such that exactly one of the ideals is generated in degree four or higher is not necessarily Koszul ($r=2$, $g_1=g_2= 2,$ and $d_1=2,d_2 \geq 4$).
\end{example}

\begin{proposition}\label{prop:KoszulCases}
Consider two collections of natural numbers $1 \leq g_1 \leq g_2 \leq \dots \leq g_r$ and $d_1, \dots, d_r$ satisfying $1 \leq d_i\leq d_j$ whenever both $i<j$ and $g_i=g_j$.

If the multi-Rees algebra of $I_1, \dots, I_r$ is Koszul for any collection of strongly stable ideals $I_1, \dots, I_r$ where each $I_i$ possesses exactly $g_i$ Borel generators of degree $d_i$, then one of the following must be true.
\begin{enumerate}
    \item $r=2$, $g_1=g_2=2$ and $2 \leq d_1 \leq d_2 \leq 3$
    \item $r>2$, $g_1=g_2=\dots=g_{r-2}=1$, $g_{r-1}=g_{r}=2$ and  $2 \leq d_{r-1} \leq d_r \leq 3$
    \item $r\geq 2$, $g_1=g_2=\dots=g_{r-1}=1$ and $g_r\leq 2$
\end{enumerate}
\end{proposition}

\begin{proof}
Combination of the observations from Examples \ref{boundnumberideals}, \ref{boundondegree}, and \ref{secondboundondegree} leads to the conclusion.
\end{proof}

In the remainder of the paper, we obtain a partial converse of \Cref{prop:KoszulCases} concerning a subcase of (1): $r=2$, $g_1=g_2=2$ and $d_1=d_2=2$. In particular, we show  that the multi-Rees algebra of two strongly stable ideals with two quadratic Borel generators is Koszul.

\section{Strongly Stable Ideals and Their Fiber Graphs}\label{sec:5}

In this section, we focus on strongly stable ideals with two quadratic Borel generators. In particular, we investigate the toric ideal of the special fiber $\mathcal{F}(I)$ associated to the ideal $I=\MB(M,N)$ where $M=x_ax_b$ and  $N=x_cx_d$ such that $c<a \leq b <d$.  Our main objective in this section is to study \emph{fiber graphs} of these ideals to obtain a quadratic Gr\"obner basis for their toric ideals with respect to two different monomial orders. Results of this section forms the foundation of the next section in which we investigate the Koszulness of the multi-Rees algebras of strongly stable ideals with two quadratic Borel generators via fiber graphs.

In \Cref{sec:directedGraph}, we provided a combinatorial method to investigate whether a collection of marked binomials $\mathcal{G}=\{g_1, \dots, g_s\}$ form a Gr\"obner basis. More specifically, we introduced the notion of directed graph of a monomial with respect to  $\mathcal{G}$ and showed that   $\mathcal{G}$ is a Gr\"obner basis if and only if the directed graph of any monomial has a unique sink and has no cycles.  One could further investigate whether  $\mathcal{G}$  is a Gr\"obner basis of a particular ideal through related combinatorial objects called fiber graphs.  The underlying idea of fiber graphs is first introduced in \cite{S} and the notion of fiber graphs is used in \cite{blasiak2008toric} to study Gr\"obner basis of toric ideals of graphic matroids.  These objects are  further developed in \cite{dipasquale2019rees, dipasquale2020koszul, schweig2011toric} to study Gr\"obner bases of different classes of toric ideals.

\begin{definition}\label{multidegree}
Let $\{u_i: i \in \mathcal{P}\}\subseteq S$ be a finite collection of monomials of the same degree in the polynomial ring $S=\KK[x_1,\dots, x_n]$.  Let $I$ be the ideal generated by the given collection. Consider the toric map $$\begin{array}{rcl}  \phi_I: \KK[T_{u_i}:i \in \mathcal{P}] & \rightarrow & \KK[u_i:i \in \mathcal{P}] \\  T_{u_i} & \rightarrow & u_i  \end{array}$$ 

and extended algebraically. The kernel of the toric map $\phi_I$ is the toric ideal of $I$ which we denote by $T(I)$. Note that the ring $R= \KK[T_
{u_i}:i \in \mathcal{P}]$ inherits the multigrading from $S$ where $R= \bigoplus_{\mu} R_{\mu}$ with $\mu$ ranging over monomials of $S$ and $R_{\mu}$ being the $\KK$-vectors space described as
$$R_{\mu} = \spa_{\KK} \{T= \prod T_{u_i}^{a_{u_i}} \in R ~:~ \phi_I (T)= \mu\}.$$
For the remainder of the paper, we abuse notation by referring to monomials of $S$ as multidegrees. 
\end{definition}

In what follows, we recall the  fiber graph construction from \cite{blasiak2008toric, schweig2011toric} (see also \cite{dipasquale2020koszul, dipasquale2019rees}). These objects are defined in terms of fibers of a multidegree under the toric map and a collection of marked binomials in $R$.

\begin{definition}\label{fibergraphmulti}
Let $\mu \in S$ be a multidegree and  $\mathcal{G}$ be a collection of marked binomials in $R$, such that the marked monomial is the initial term of the binomial with respect to a given term order $>$.  The \emph{fiber graph of $I$ at $\mu$ with respect to $\mathcal{G}$}, denoted by $\Gamma_{\mu}(I)$, is defined as follows: the vertices of $\Gamma_{\mu}(I)$ correspond to monomials in $R_{\mu}$, i.e., a monomial $T$ is a vertex if and only if $\phi_{I}(T)=\mu$. For each pair of vertices  $T$ and $T'$, there is a directed  edge from $T$ to $T'$ whenever $T \rightarrow_{\mathcal{G}} T'$. For the sake of simplicity, we suppress the collection of marked binomials   $\mathcal{G}$ in the notation for $\Gamma_{\mu}(I)$ as our choice of $\mathcal{G}$ can be understood from context.
\end{definition}

\begin{remark}
It follows from the construction of fiber graphs  that for each $T \rightarrow_{\mathcal{G}} T'$, the monomial $T $ comes before the monomial $T'$  with respect to a term order $\succ_{\tau}$ such that the marked monomials of $\mathcal{G}$ are the initial terms with respect to $\succ_{\tau}$.  Thus, none of the fiber graphs  possess  cycles. 
\end{remark} 

Fiber graphs are significant tools to study  Gr\"obner bases of $T(I)$ and this importance is due to the following corollary which is a direct consequence of 
 \Cref{fibergraphinduces} and the definition of fiber graphs (see also \cite[Theorem 5.5]{S}, \cite[Proposition 4.5]{dipasquale2019rees} and \cite[Proposition 2.5]{dipasquale2020koszul}).

\begin{corollary}\label{fibergraphdefinesGB}
The collection of marked binomials $\mathcal{G}\subseteq R$ is a Gr\"obner basis for $T(I)$ if and only if the fiber graph $\Gamma_{\mu}(I)$ is either empty or has a unique sink for every multidegree $\mu \in S$.  
\end{corollary}

\begin{remark}\label{directsumfibergrapgGB}
Notice that the fiber graph construction given in  \Cref{fibergraphmulti} is associated to a toric map $\phi_I$. Thus one can define fiber graphs for direct sum of ideals $I_1 \oplus \dots \oplus I_r$ via the toric map  $\varphi '$ from \Cref{def:multi}. As a result, we can extend  \Cref{fibergraphdefinesGB} to show that a collection of marked binomials form a Gr\"obner basis for $T(I_1 \oplus \dots \oplus I_r)$.
\end{remark}

\begin{notation}\label{abcd}
 For the remainder of this section, we use the notation  $I=\MB (M,N)$ to denote a strongly stable ideal with two  quadratic Borel generators $M$ and $N$ where $M=x_ax_b$, $N=x_cx_d$ such that  $c<a \leq b <d$. Furthermore, we denote the collection of the minimal monomial generators of $\MB (M)$ by $\MB_M$ and collection of the minimal monomial generators in $\MB (N) \setminus \MB (M)$ by $\MB_N$. 
\end{notation}

In certain cases, one can classify vertices of a fiber graph based on the factorization of a vertex. In such situations, the structure of one vertex determines the structure of all other vertices of a fiber graph. Before we provide this classification, we introduce the following notation. 

\begin{definition}
A vertex $T$ in $\Gamma_{\mu}(I)$ is  called a  \emph{type $M$ vertex} if $ T= \prod_{i=1}^s T_{m_i}$ where $m_i \in \MB_M$ for all $i \in [s]$. Similarly,  if $T= \prod_{j=1}^t T_{n_j} $ where $ n_j \in \MB_N$ for all $j\in [t]$, we say $T$ is a  \emph{type $N$ vertex}.
\end{definition}

A fiber graph may have vertices which are neither type $M$ nor type $N$. However, the existence of a type $M$ vertex forces
all other vertices to also be of type $M$. The same situation occurs when a fiber graph has a vertex of type $N$.

\begin{lemma}\label{vertices}
Let $\mu$ be a multidegree in $S$ and $I=\MB (M,N)$.  
\begin{enumerate}[(a)]
    \item If $ \Gamma_{\mu}(I)$ has a vertex of type $M$, then all other vertices of $ \Gamma_{\mu}(I)$  are of type $M$.
       \item If $ \Gamma_{\mu}(I)$ has a vertex of type $N$, then all other vertices of $ \Gamma_{\mu}(I)$ are of type $N$.
\end{enumerate}
\end{lemma}

\begin{proof}
(a) Let $ T= T_{m_1} \cdots T_{m_k}$ where $m_i=x_{s_i}x_{t_i} \in \MB_M$. Since $T$ is a vertex in $\Gamma_{\mu} (I)$, we have 
$\mu = m_1\cdots m_k= \prod_{i=1}^k x_{s_i}x_{t_i}$
where $s_i \leq t_i \leq b$ for each $i$. Note that there exists no $x_q$ dividing $\mu$ such that $q>b$.

Suppose there exists another vertex $T'$ of $\Gamma_{\mu} (I)$ such that $T'= \Big( \prod_{i=1}^{k'} T_{f_i}\Big)  \Big(  \prod_{j=1}^{l} T_{n_j}\Big)$ where $f_i \in \MB_M$ for each $i \in [k']$ and $n_j \in \MB_N$ for each $j\in [l]$. Let $f_i= x_{a_i}x_{b_i}$ and $n_j= x_{p_j}x_{q_j}$ for each $i$ and $j$.  Then
 $$\mu =   \Big( \prod_{i=1}^{k'} x_{a_i}x_{b_i} \Big) \Big( \prod_{j=1}^{l}x_{p_j}x_{q_j} \Big)$$
 where $a_i \leq b_i\leq b$ and $p_j \leq c< b<q_j$ for each $i$ and $j$. Since $\mu$ does not have any factor $x_q $ with $q>b$, we must have $l=0$. Thus $k=k'$ and $T'$ is a vertex of type $M$.

(b) Let $ T= T_{n_1} \cdots T_{n_l}$ where $n_i=x_{p_i}x_{q_i} \in \MB_N$ such that $p_i \leq c<b<q_i$ for each $i$. Since $T$ is a vertex in $\Gamma_{\mu} (I)$, we have  $\mu = n_1\cdots n_l$.    Note that any factor of $\mu$ divisible only by variables with indices greater than $b$ is necessarily a divisor of $x_{q_1}\cdots x_{q_l}$.

Suppose there exists another vertex $T'$ of $\Gamma_{\mu} (I)$ such that $T'= \Big( \prod_{i=1}^{k} T_{m_i}\Big)  \Big(  \prod_{j=1}^{l'} T_{g_j}\Big)$ where $m_i \in \MB_M$ for each $i \in [k]$ and $g_j \in \MB_N$ for each $j\in [l']$.  Let $m_i= x_{s_i}x_{t_i}$ and $q_j= x_{c_j}x_{d_j}$ for each $i$ and $j$.  Then
 $$\mu = \Big( \prod_{i=1}^{k} x_{s_i}x_{t_i} \Big) \Big( \prod_{j=1}^{l'}x_{c_j}x_{d_j} \Big)$$
 where $s_i \leq t_i\leq b$ and $c_j \leq c< b<d_j$ for each $i$ and $j$. Then we must have $x_{q_1} \cdots x_{q_l} = x_{d_1} \cdots x_{d_{l'}}$ which implies that $l=l'$. Thus $k=0$ and $T'$ is a vertex of type $N$.
\end{proof}

\begin{notation}
 For the sake of simplicity, we often use the notation $T_{ij}$ to refer to  $T_{x_ix_j}$ for $i\leq j$. It should also be noted that  $T_{ij}=T_{x_ix_j}=T_{x_jx_i}=T_{ji}$. We set $T_{ij}$ as our standard notation where $i\leq j$.
\end{notation}

\subsection{The revlex Gr\"obner basis}\label{subsec:RevLex} 

In this subsection, we describe a Gr\"obner basis for $T(I)$ where $I= \MB (M,N)$ with respect to graded reverse lexicographic order by using fiber graphs. 
\begin{definition}\label{def:MonRevLex}
We say $T_{ij}>_{rlex} T_{i'j'}$  if $x_ix_j\succ_{rlex} x_{i'}x_{j'}$ in $S=K[x_1,\dots,x_n]$. We use $\succ_{rlex}$ to also denote the graded reverse lexicographic order in $R$ induced by $>_{rlex}$ on the $T$ variables.
\end{definition}

The following result was first proved in \cite[Theorem~1.3]{Conca} for symmetric ladder ideals (see \Cref{ladder}). Our proof differs from \cite{Conca} as it employs fiber graphs to obtain a given collection as a Gr\"obner basis. Additionally, the ideas used in the proof of the following lemma are fundamental to establish our last main result (\Cref{thm:MultiReesGB}).

\begin{theorem}\label{grobnerrevlex} 
Let $\mathcal{G}_1=\{\underline{T_{u}T_{v}}-T_{u'}T_{v'}: uv=u'v' \text{ and } u,v \succ_{rlex} v'\}$. Then $\mathcal{G}_1$ is a Gr\"obner basis of $T(I)$ with respect to the term order $\succ_{rlex}$. 
\end{theorem}

\begin{proof} 
It suffices to prove that, for each multidegree $\mu \in S$ with respect to $\mathcal{G}_1$, the fiber graph $\Gamma_{\mu}(I)$ is empty  or has a unique sink by \Cref{fibergraphdefinesGB}. Thus our goal is to prove the claim: ``\textit{If} $\mu$ \textit{is a multidegree in} $S$, \textit{then} $\Gamma_{\mu}(I)$ \textit{is empty or possesses a unique sink}.'' If the degree of $\mu$ is odd, then $\Gamma_{\mu}(I)$ must be empty. Suppose the degree of $\mu$ is even. If $\Gamma_{\mu}(I)$ is empty, we are done.  If not, the proof follows from \Cref{cor:uniqueSink}.
\end{proof}

\begin{remark}\label{ladder} Consider the symmetric $d \times d$ matrix whose entry in row $i$ and column $j$ is $T_{ij}$ if $x_ix_j \in \mathcal{B}(M,N)$ and $0$ otherwise. This matrix is represented by the symmetric ladder diagram (the colored area) given in \Cref{staircasediagram}. The set $\mathcal{G}_1$ is the collection of the $2 \times 2$ minors of this matrix that correspond to binomials. In this case, the leading (marked) monomial is chosen according to the term order $\succ_{rlex}$, defined in  \Cref{def:MonRevLex}, and it corresponds to the main diagonal of the related $2 \times 2$ minor.
The variables of $R$ are ordered in the following way.
\begin{align*}
    T_{11}>_{rlex} T_{12}>_{rlex} T_{22} >_{rlex} \dots >_{rlex} T_{1b}>_{rlex}T_{2b}>_{rlex} \dots >_{rlex} T_{ab}>_{rlex}  \\
     T_{1(b+1)}>\dots >_{rlex} T_{c(b+1)}>_{rlex} \dots >_{rlex} T_{1d}>_{rlex} \dots >_{rlex}T_{cd}
\end{align*}
\begin{figure}[h]
\begin{center}
    	\begin{tikzpicture}[scale=0.7]
	
	\begin{scope}[greennode/.style={circle, draw=green!60, fill=green!5}, bluenode/.style={circle, draw=blue!60, fill=blue!5}, rednode/.style={circle, draw=red!60, fill=red!5}]
	\end{scope}
	
\draw[step=1.0,black] (0,0) grid (9,9);
\draw[red,thick,dashed] (0,9) -- (2,9);
\draw[red,thick,dashed] (2,9) -- (2,7);
\draw[red,thick,dashed] (2,7) -- (4,7);
\draw[red,thick,dashed] (4,7) -- (4,4);
\draw[red,thick,dashed] (4,4) -- (7,4);
\draw[red,thick,dashed] (7,4) -- (7,2);
\draw[red,thick,dashed] (7,2) -- (9,2);
\draw[red,thick,dashed] (9,2) -- (9,0);

\node[left] at (0,1.5) {$c$};
\node[left] at (0,3.5) {$a$};
\node[left] at (0,6.5) {$b$};
\node[left] at (0,8.5) {$d$};
\node[above] at (1.5,-0.7) {$c$};
\node[above] at (3.5,-0.7) {$a$};
\node[above] at (6.5,-0.7) {$b$};
\node[above] at (8.5,-0.7) {$d$};

\filldraw[fill=blue!20!white, draw=black] (0,0) rectangle (1,1);
\filldraw[fill=blue!20!white, draw=black] (1,0) rectangle (2,1);
\filldraw[fill=blue!20!white, draw=black] (2,0) rectangle (3,1);
\filldraw[fill=blue!20!white, draw=black] (3,0) rectangle (4,1);
\filldraw[fill=blue!20!white, draw=black] (4,0) rectangle (5,1);
\filldraw[fill=blue!20!white, draw=black] (5,0) rectangle (6,1);
\filldraw[fill=blue!20!white, draw=black] (6,0) rectangle (7,1);
\filldraw[fill=blue!20!white, draw=black] (7,0) rectangle (8,1);
\filldraw[fill=blue!20!white, draw=black] (8,0) rectangle (9,1);
\filldraw[fill=blue!20!white, draw=black] (1,1) rectangle (2,2);
\filldraw[fill=blue!20!white, draw=black] (2,1) rectangle (3,2);
\filldraw[fill=blue!20!white, draw=black] (3,1) rectangle (4,2);
\filldraw[fill=blue!20!white, draw=black] (4,1) rectangle (5,2);
\filldraw[fill=blue!20!white, draw=black] (5,1) rectangle (6,2);
\filldraw[fill=blue!20!white, draw=black] (6,1) rectangle (7,2);
\filldraw[fill=blue!20!white, draw=black] (7,1) rectangle (8,2);
\filldraw[fill=blue!20!white, draw=black] (8,1) rectangle (9,2);
\filldraw[fill=blue!20!white, draw=black] (2,2) rectangle (3,3);
\filldraw[fill=blue!20!white, draw=black] (3,2) rectangle (4,3);
\filldraw[fill=blue!20!white, draw=black] (4,2) rectangle (5,3);
\filldraw[fill=blue!20!white, draw=black] (5,2) rectangle (6,3);
\filldraw[fill=blue!20!white, draw=black] (6,2) rectangle (7,3);
\filldraw[fill=blue!20!white, draw=black] (3,3) rectangle (4,4);
\filldraw[fill=blue!20!white, draw=black] (4,3) rectangle (5,4);
\filldraw[fill=blue!20!white, draw=black] (5,3) rectangle (6,4);
\filldraw[fill=blue!20!white, draw=black] (6,3) rectangle (7,4);

\filldraw[fill=red!40!white, draw=black] (0,8) rectangle (1,9);
\filldraw[fill=red!40!white, draw=black] (1,8) rectangle (2,9);
\filldraw[fill=red!40!white, draw=black] (0,7) rectangle (1,8);
\filldraw[fill=red!40!white, draw=black] (1,7) rectangle (2,8);
\filldraw[fill=red!40!white, draw=black] (0,6) rectangle (1,7);
\filldraw[fill=red!40!white, draw=black] (1,6) rectangle (2,7);
\filldraw[fill=red!40!white, draw=black] (2,6) rectangle (3,7);
\filldraw[fill=red!40!white, draw=black] (3,6) rectangle (4,7);
\filldraw[fill=red!40!white, draw=black] (3,5) rectangle (4,6);
\filldraw[fill=red!40!white, draw=black] (2,5) rectangle (3,6);
\filldraw[fill=red!40!white, draw=black] (1,5) rectangle (2,6);
\filldraw[fill=red!40!white, draw=black] (0,5) rectangle (1,6);
\filldraw[fill=red!40!white, draw=black] (3,4) rectangle (4,5);
\filldraw[fill=red!40!white, draw=black] (2,4) rectangle (3,5);
\filldraw[fill=red!40!white, draw=black] (1,4) rectangle (2,5);
\filldraw[fill=red!40!white, draw=black] (0,4) rectangle (1,5);
\filldraw[fill=red!40!white, draw=black] (2,3) rectangle (3,4);
\filldraw[fill=red!40!white, draw=black] (1,3) rectangle (2,4);
\filldraw[fill=red!40!white, draw=black] (0,3) rectangle (1,4);
\filldraw[fill=red!40!white, draw=black] (1,2) rectangle (2,3);
\filldraw[fill=red!40!white, draw=black] (0,2) rectangle (1,3);
\filldraw[fill=red!40!white, draw=black] (0,1) rectangle (1,2);
	\end{tikzpicture}

\caption{The symmetric ladder diagram for $I=\mathcal{B}(x_ax_b,x_cx_d)$ with $c<a \leq b<d$}    
\label{staircasediagram}
    \end{center}
\end{figure}
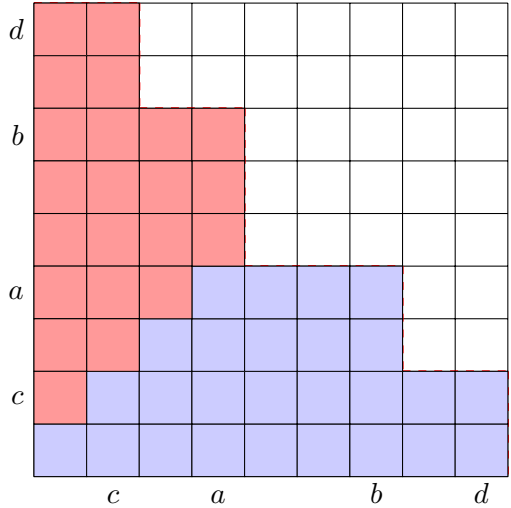

\end{remark}

In order to study fiber graphs in a systematic way, we set the following standard factorization for a monomial in $R$ with respect to the $\succ_{rlex}$ order.

\begin{notation}\label{lowestdivisor} Let $T$ be a monomial in $R$. We express  $T$ as 
$T=(T_{p_1}\cdots T_{p_k})(T_{p_{k+1}} \cdots T_{p_{k+l}})$ such that 
\begin{itemize}
    \item $p_i \in \MB_M$ for each $i \in \{1, \ldots, k\}$ and $p_{k+j} \in \MB_N$ for each $j \in \{1, \ldots, l\}$,  
    \item $p_1  \succeq_{rlex} \dots \succeq_{rlex} p_k \succeq_{rlex} p_{k+1} \succeq_{rlex} \dots \succeq_{rlex} p_{k+l}$.
\end{itemize} 
In the above expression, we denote the latest variable factor of $T$ in $\MB_M$ by $L_M (T)$ and the latest variable factor of $T$ in $\MB_N$ by $L_N (T)$.  Based on the above expression of $T$, one has $L_M (T)=T_{p_k}$ and $L_N (T)=T_{p_{k+l}}$ unless $k$ or $l$ is equal to zero. 

Given a multidegree $\mu$  in $S$, we use $M'$ and $N'$ to denote the smallest elements in $\MB_M$  and $\MB_N$, respectively, with respect to $\succ_{rlex}$ dividing $\mu$.
\end{notation}

In the following example, we provide an example of a fiber graph and we use the above standard factorization to label the vertices of the fiber graph.

\begin{example}\label{ex:33_25}
Let $I=\mathcal{B}(x_3^2, x_2x_5)$ and let $\mu=x_1^2x_2^2x_3^2x_4x_5.$ Then the fiber graph $\Gamma_{\mu}(I)$  is given in \Cref{f:fibergraph}.

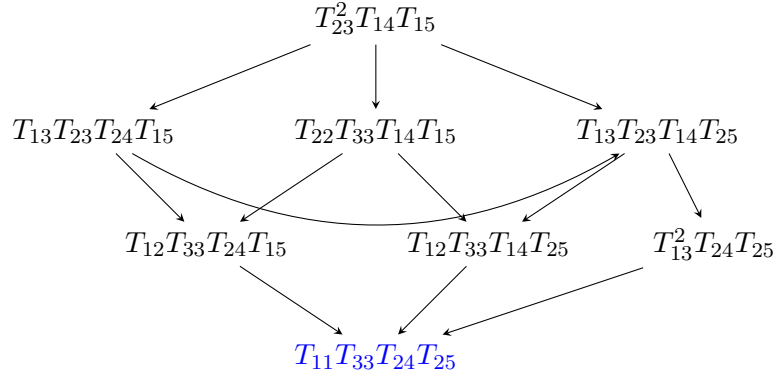
\begin{figure}[h]
  \begin{center}
    \begin{tikzpicture}[scale=0.75,->,>=stealth]
      \node (0) at (0,0) {$T_{23}^2T_{14} T_{15}$};
      \node (1) at (-5,-2) {$T_{13}T_{23}T_{24} T_{15}$};
      \node (2) at (0,-2) {$T_{22}T_{33}T_{14} T_{15}$};
      \node (3) at (5,-2) {$T_{13}T_{23}T_{14} T_{25}$};
      \node (4) at (-3,-4) {$T_{12}T_{33}T_{24} T_{15}$};
      \node (5) at (2,-4) {$T_{12}T_{33}T_{14} T_{25}$};
      \node (6) at (6,-4) {$T_{13}^2T_{24} T_{25}$};
      \node (7) at (0,-6) {\textcolor{blue}{$T_{11}T_{33}T_{24} T_{25}$}};

    \draw (0)--(1);
        \draw (0)--(2);
            \draw (0)--(3);
    \draw (1)--(4);
        \draw (1) to [bend right] (3);
    \draw (2)--(4);
        \draw (2)--(5);
            \draw (3)--(5);
        \draw (3)--(6);
        \draw (4)--(7);
        \draw (5)--(7);       
        \draw (6)--(7);            
    \end{tikzpicture}
    \caption{The fiber graph $\Gamma_{x_1^2x_2^2x_3^2x_4x_5} (I)$ for $I=\mathcal{B}(x_3^2, x_2x_5)$.}    
\label{f:fibergraph}
    \end{center}
\end{figure}

Note that the fiber graph given in \Cref{f:fibergraph} has no cycles and has a unique sink $T= T_{11}T_{33}T_{24} T_{25}$ where $L_M (T)= T_{33}$ and $L_N (T)= T_{25}$. In addition, both  Borel generators $M=x_3^2$ and $N=x_2x_5$ divides $\mu$, so  $M=M'$ and $N=N'$. As we shall see in the following lemma, it is not a coincidence that $L_N(T)=T_{25}=T_N$ for the sink vertex $T$.
\end{example}

\begin{lemma}\label{fibergraphmovesforrlex} 
(Adopt \Cref{lowestdivisor}) Let $\mu$ be a multidegree and  $T$ be a vertex of $\Gamma_{\mu}(I)$.   
 \begin{enumerate}[(a)]
    \item If $L_N (T)$ exists and $L_N (T) \neq T_{N'}$, then there is an edge from $T$ to a vertex $T'$ of $\Gamma_{\mu}(I)$ such that
     $$L_N (T) >_{rlex}  L_N (T') = T_{N'}.$$
    \item Suppose $L_N (T)$ does not exist. If $L_M (T)$ exists and $L_M (T) \neq T_{M'}$ then there is an edge from $T$ to a vertex $T'$ of $\Gamma_{\mu}(I)$  such that 
    $$L_M (T) >_{rlex}  L_M (T') =  T_{M'}.$$
\end{enumerate}
\end{lemma}

    \begin{proof} (a) Let $L_N (T)= T_{ij}$ and $N'=x_{i'}x_{j'}$ where $x_ix_j$ and $x_{i'}x_{j'}$ are both in $\MB_N$ with $i\leq j$ and $i'\leq j'$.  It follows from our assumption $L_N (T) \neq T_{N'}$ that  $ i < i'$  or  $j<j'$.  We start by observing $j=j'$. Suppose not. Since  $x_{j'}$ divides $\mu$,  there is some  monomial $u= x_s x_{j'}$ such that $T_u$ divides $T$. Then $T_{ij}>_{rlex}  T_u$, contradicting how $T_{ij}$ was chosen. Thus $j=j'$ and  $i< i'$.

    Our goal is to find $u, u' \in \mathcal{B} (M,N)$ such that $T_u T_{ij} -T_{u'} T_{i'j} \in \mathcal{G}_1$ where $T_uT_{ij}$ divides $T$. Once we find such $u, u'$, we can construct a vertex $T'$ of $\Gamma_{\mu}(I)$  such that $T \rightarrow_{\mathcal{G}_1} T'$. Since $x_{i'}$ divides $\mu$, there exists a monomial $u \in \mathcal{B} (M,N)$ divisible by $x_{i'}$ so that $T_u$ divides $T$. We claim that $T_u T_{ij}$ divides $T$. In order to prove the claim, it suffices to show $T_u \neq T_{ij}$. If $u=x_ix_j$, we must have  $i'=j$ as $u$ is divisible by $x_{i'}$ and $i<i'$.   Additionally, since $x_{j'}$ divides $\mu$, there exists a monomial $v \in \mathcal{B} (M,N)$ divisible by $x_{j'}$ such that $T_v$ divides $T$. Note that $T_v >_{rlex} L_N (T)=T_{ii'}$  by the definition of $L_N (T)$.  It follows from this comparison that $i'=j'$. This is not possible because $i'<j'$ since $x_{i'} x_{j'} \in \MB_N$.  Therefore, $T_uT_{ij}$ divides $T$.

  Let $u':=x_{i} ( u/x_{i'}) $ and $\displaystyle T' := T_{u'} T_{i'j} \big( T / (T_u T_{ij}) \big)$. It follows from the definition of strongly stable ideals that $u'\in \MB (M,N)$ as $i<i'$. Note that $T'$ is a vertex of $\Gamma_{\mu}(I)$.  Since $T_u >_{rlex} T_{ij} >_{rlex} T_{i'j}$, we have  $T_u T_{ij} -T_{u'} T_{i'j} \in \mathcal{G}_1$. Therefore,  $T \rightarrow_{\mathcal{G}_1} T'$ and  $L_N (T')= T_{N'}= T_{i'j}$.

(b) First note that $T$ is a type $M$ vertex because $L_N(T)$ does not exist.  Let $L_M (T)= T_{ij}$ and $M'= x_{i'}x_{j'}$ where $x_ix_j$ and $x_{i'}x_{j'}$ are both in $\MB_M$ with $i\leq j$ and $i'\leq j'$. It follows from our assumption $L_M (T) \neq T_{M'}$ that  $ i < i'$  or  $j<j'$.  Using the same arguments from (a), we have  $j=j'$ and $i< i'$.

Similar to the discussion in (a), we claim that there exists a monomial $u \in \MB_M$ such that it is divisible by $x_{i'}$ and $T_uT_{ij}$ divides $T$.  Since $x_{i'}$ divides $\mu$, there exists a monomial $u$ is divisible by $x_{i'}$ such that $T_u$ divides $T$. If $T_u \neq T_{ij},$ then $T_uT_{ij}$ divides $T$. If $T_u=T_{ij}$, then $u= x_ix_j$. Since $i<i'$ and $u$ is divisible by $x_{i'}$, we must have $i'=j=j'$ which implies that $M'=x_{i'}^2$ and $T_u=L_M (T)= T_{ii'}$. Since $M'$ divides $\mu,$ there exists another $T_{v}$ in the support of $T$ such that $v$ is divisible by $x_{i'}$. Then $T_vT_{ij}$ divides $T$. Since  $L_N(T)$ does not exist, $u,v$ must be both in $\MB_M$ and the claim holds. Since $u \in \MB_M$ and $x_{i'}$ divides $u$ where $i<i'$, the monomial  $u':=x_{i} (u / x_{i'})$ is in $\MB_M.$ Recall also that $M'= x_{i'} x_j \in \MB_M$. Then,  $T_u T_{ij} -  T_{u'} T_{i'j} \in \mathcal{G}_1$ which implies that $T \rightarrow_{\mathcal{G}_1} T'$ where 
 $\displaystyle T' = T_{u'} T_{i'j}  \big( T /(T_u T_{ij}) \big)$ is a vertex in $\Gamma_{\mu} (I)$. Note that $L_M (T')= T_{i'j}$ where  $L_M (T) >_{rlex} L_M (T') =T_{M'}$.
 \end{proof}

\begin{lemma}\label{sinkreductionrlex}
(Adopt \Cref{lowestdivisor}) Let $\mu$ be a multidegree such that $\Gamma_{\mu}(I)$ is nonempty.  If the fiber graph $\Gamma_{\mu}(I)$ has a vertex of type $M$, then any sink of $\Gamma_{\mu}(I)$ is of type $M$. In particular, every sink of $\Gamma_{\mu}(I)$ is of the form $Z_{M'}T_{M'}$ where $Z_{M'}$ is a sink of $\Gamma_{\frac{\mu}{M'}}(I)$. Otherwise, each sink of $\Gamma_{\mu}(I)$ is of the form $Z_{N'}T_{N'}$ where $Z_{N'}$ is a sink of $\Gamma_{\frac{\mu}{N'}}(I)$.
\end{lemma}

\begin{proof} Recall that whenever there is a directed edge from $T$ to $T'$, we have $T \succ_{rlex} T'$. Additionally, observe that a sink of $\Gamma_{\mu}(I)$ must be reached eventually because there is no infinite descending chain of monomials with respect to a monomial order (see \cite[Lemma 2.1.7]{herzog2011monomial}).

If there is a vertex $T$ of type $M$  in $\Gamma_{\mu}(I)$, then every vertex of $\Gamma_{\mu}(I)$ must be of type $M$ by Lemma \ref{vertices} (a). Thus any sink is of type $M$.  Furthermore, every sink of $\Gamma_{\mu}(I)$ must be divisible by $T_{M'}$. Otherwise, there is an edge directed from a sink by \Cref{fibergraphmovesforrlex} (b), a contradiction.  Hence each sink is of the form $Z_{M'}T_{M'}$ where $Z_{M'}$ is a vertex in $\Gamma_{\frac{\mu}{M'}}(I)$. Finally, the monomial $Z_{M'}$ must be a sink in $\Gamma_{\frac{\mu}{M'}}(I)$. Otherwise, there must be an edge directed from $Z_{M'}$ to another vertex in $\Gamma_{\frac{\mu}{M'}}(I)$, say  $Z'$. Then, $Z_{M'}\rightarrow_{\mathcal{G}_1} Z'$ which in turn implies that $Z_{M'}T_{M'} \rightarrow_{\mathcal{G}_1} T_{M'}Z'$, contradicting the fact that  $Z_{M'}T_{M'}$ is a sink in $\Gamma_{\mu}(I)$. 

If $\Gamma_{\mu}(I)$ has no vertices of type $M$, then $L_N (T)$ exists for each vertex $T$ of $\Gamma_{\mu}(I)$. Furthermore,  every sink of $\Gamma_{\mu}(I)$ must be divisible by $T_{N'}$. Otherwise, it could not be a sink due to \Cref{fibergraphmovesforrlex} (a). Therefore, each sink is of the form $Z_{N'}T_{N'}$ where $Z_{N'}$ is a vertex in $\Gamma_{\frac{\mu}{N'}}(I)$. As in the previous paragraph, one can show that  $Z_{N'}$ must be a sink in $\Gamma_{\frac{\mu}{N'}}(I)$. 
\end{proof}

\begin{remark}
Let $\mu$  be a multidegree. If the degree of $\mu$ is odd, then $\Gamma_{\mu}(I)$ must be empty. Thus, we may assume the degree of $\mu$ is even for the remainder of the paper.
\end{remark}

In the following corollary, we show that a non-empty fiber graph has a unique sink and we provide a greedy description of this unique sink. In summary, the sink is found as follows: set $p_r$ to be $N'$ if $N'$ exists; otherwise, set $p_r= M'$. Let $\mu_{r-1}=\mu / p_r$. Furthermore,  let $M'_{r-1}$ and $N'_{r-1}$ be the smallest elements in $\MB_M$ and $\MB_N$ in revlex order, respectively, dividing $\mu_{r-1}$. Set $p_{r-1}$ equal to $N'_{r-1}$ if it exists; otherwise, set $p_{r-1}=M'_{r-1}$.  One can continue in this fashion by setting $\mu_{i}=\mu / \prod_{j=i+1}^r p_j$ and letting $M'_i$ and $N'_i$ be the smallest elements in $\MB_M$ and $\MB_N$ in revlex order, respectively,  dividing $\mu_{i}$ for each $i \in [r-1]$.

\begin{corollary}\label{cor:uniqueSink}
For any  multidegree $\mu$ of degree $2r$, where $\Gamma_{\mu}(I)$ is non-empty, $\Gamma_{\mu}(I)$  has a unique sink $T= T_{p_1} \cdots T_{p_r}$ satisfying that  $T_{p_i}$  is the least variable in revlex order so that $p_i$ divides $\mu /  \prod_{j=i+1}^r p_j$ for each $i\in [r]$.
\end{corollary}
\begin{proof}
We use induction on  the length of the standard factorization. The base case $r=1$ is immediate. Suppose $r>1$. Then a sink $T= T_{p_1} \cdots T_{p_r}$ in $\Gamma_{\mu}(I)$ is of the form $T=T' ~T_{M'}$ or $T=T' ~T_{N'}$ such that $T'=T_{p_1} \cdots T_{p_{r-1}}$ is a sink in $\Gamma_{\mu_{r-1}} (I)$ by \Cref{sinkreductionrlex} and its proof. Note that $T$ is of the first form, i.e. $p_r=M'$, if $\Gamma_{\mu}(I)$ has a vertex of type $M$ and it is of the second form, i.e. $p_r=N'$, otherwise. It follows from the induction hypothesis that  $T'$ is the unique sink in $\Gamma_{\mu_{r-1}} (I)$ and $T_{p_i}$  is the smallest variable in revlex order so that $p_i$  divides $\mu /  \prod_{j=i+1}^r p_j$ for each $i \in [r-1]$. Since $\Gamma_{\mu_{r-1}} (I)$ has a unique sink, $T$ is the only sink in $\Gamma_{\mu}(I)$. Hence, the statement holds.
\end{proof}

We conclude this subsection by presenting an explicit description of the relations between the indices of the variables that are involved in the standard factorization of a sink. These relations will play a central role in the proof of the main result in \Cref{sec:6}. Furthermore, they can be used to count the number of different fiber graphs for a given degree.

\begin{lemma}\label{TailComparisonRevlex}
(Adopt \Cref{abcd} and \Cref{lowestdivisor}) Let $\Gamma_{\mu}(I)$ be a nonempty fiber graph and $T=T_{p_1}\cdots T_{p_r}$ be its unique sink.  Let $p_{k_1}=x_ix_j$ and $p_{k_2}=x_{i'}x_{j'}$ where $k_1 <k_2$.  
\begin{enumerate}
    \item[(a)] If $p_{k_1}$ and $p_{k_2}$ are both in $\MB_M$ or $\MB_N$, then $i \leq i'$ and $j \leq j'$.
    \item[(b)] If $p_{k_1} \in \MB_M$ and $p_{k_2}\in \MB_N$, then $i\leq i'$ or $c<i$. In particular, we must have $i=j$ or $c<j$  when $i=i'$.
\end{enumerate}
\end{lemma}
\begin{proof} (a) The last inequality $j \leq j'$ follows from our assumption that $p_{k_1} \succeq_{rlex} p_{k_2}$. For the first inequality, by contradiction, suppose $i>i'$. Observe that we must have $j<j'$ as $p_{k_1}\succ_{rlex} p_{k_2}$. Before we proceed further recall that $i\leq j$ and $i' \leq j'$. If $i'=j'$, then $p_{k_2} \in \MB_M$ as the equality of two indices is not possible for monomials in $\MB_N$. Then we obtain $i'< i \leq j \leq j'$ which implies that $i' < j'$, a contradiction. Thus, we must have  $i \leq i'$ in this case. For the remainder of the proof, we may assume that $i'<j'$.  Note that  $x_{i'}x_j$ and $x_ix_{j'}$ are both in $\MB_M$ if $p_{k_1},p_{k_2} \in \MB_M$. Similarly,   they are both in $\MB_N$ if $p_{k_1},p_{k_2} \in \MB_N$. Furthermore, $p_{k_1},p_{k_2} \succ_{rlex}  x_ix_{j'}$. Then $T_{p_{k_1}}T_{p_{k_2}} - T_{i'j}T_{ij'} \in \mathcal{G}_1$, which contradicts to $T$ being a sink. Therefore, $i \leq i'$.

(b) By contradiction, suppose $i'<i\leq c.$ It follows from the definition of strongly stable ideals that $x_{i'}x_j \in \MB_M$ as  $x_{i} x_{j}\in \MB_M$ where $i'<i$. Furthermore, we have $x_{i}x_{j'} \in \MB_N$ as $i\leq c$ and $p_{k_2} \in \MB_N$. Since $p_{k_1},p_{k_2}\succ_{rlex} x_{i}x_{j'}$, we have $T_{p_{k_1}}T_{p_{k_2}} - T_{i'j}T_{ij'} \in \mathcal{G}_1$, a contradiction.

For the last statement, assume that $i=i'$. Since $p_{k_1} \succ_{rlex} p_{k_2},$ it is immediate that $j<j'$. On the contrary, suppose $i < j \leq c$. Then, we have $x_jx_{j'} \in \MB_N$ as $j \leq c$. Moreover,  $x_i^2 \in \MB_M$ since $x_ix_j \in \MB_M$ and $i<j$. Note that $p_{k_1},p_{k_2} \succ_{rlex} x_jx_{j'}$. Thus $T_{p_{k_1}}T_{p_{k_2}} - T_{ii}T_{jj'} \in \mathcal{G}_1$, a contradiction.
\end{proof}

\begin{remark}
As it can be seen from the proof of the above lemma, uniqueness of the sink is not necessary for the lemma to hold.  
\end{remark}

\subsection{The mixed revlex Gr\"obner basis}\label{subsec:Mixed}  Our goal in this subsection is to describe a Gr\"obner basis of $T(I)$ with respect to a new order called \emph{mixed reverse lexicographic order}. We use similar approaches to the ones employed in the previous subsection and describe the differences in the structure of fiber graphs with respect to a new collection $\mathcal{G}_2$.

\begin{definition}\label{def:MonMixRevLex}
Consider all the monomials in $\MB_M \cup \MB_N.$ We say $m \succ_{mrlex} n$  if one of the following hold:
\begin{enumerate}
    \item[(i)] $m \succ_{rlex} n$ and $m,n \in \MB_M$,
    \item[(ii)] $m \in \MB_N$ and $n \in \MB_M$,
     \item[(iii)] $m \succ_{rlex} n$ and $m, n \in \MB_N$.
\end{enumerate}
We say $T_{m}>_{mrlex} T_{n}$ if and only if $m \succ_{mrlex} n$. Let $a,b,c$ and $d$ be given as in \Cref{abcd}. The $T$ variables are ordered as follows:
\begin{align*}
T_{1(b+1)}>_{mrlex}\dots >_{mrlex}T_{c(b+1)}>_{mrlex}\dots >_{mrlex}T_{1d}>_{mrlex}\dots\\
>_{mrlex}T_{cd}>_{mrlex}T_{11}>_{mrlex}T_{12}>_{mrlex}T_{22}>_{mrlex}\dots \\
>_{mrlex}T_{1b}>_{mrlex}T_{2b}>_{mrlex}\dots >_{mrlex}T_{ab}.
\end{align*}
We use the notation $\succ_{mrlex}$ to denote the reverse lexicographic order induced by $>_{mrlex}$ variable order on $R$. We call this new term order the \emph{mixed reverse lexicographic order}.
\end{definition}

Similar to \Cref{lowestdivisor}, we set a standard factorization to express each monomial in $R$ with respect to mixed reverse lexicographic order.   

\begin{notation}\label{Mixlowestdivisor} 
Let $T$ be a monomial in $R$. We express  $T$ as 
$T=(T_{p_1}\cdots T_{p_k})(T_{p_{k+1}} \cdots T_{p_{k+l}})$ such that 
\begin{itemize}
    \item $p_i \in \MB_N$ for each $i \in \{1, \ldots, k\}$ and $p_{k+j} \in \MB_M$ for each $j \in \{1, \ldots, l\}$,  
    \item $p_1  \succeq_{mrlex} \dots \succeq_{mrlex} p_k \succ_{mrlex} p_{k+1} \succeq_{mrlex} \dots \succeq_{mrlex} p_{k+l}$.
\end{itemize} 
We still denote  the latest variable factor of $T$ in $\MB_M$ by $L_M (T)$ and the latest variable factor of $T$ in $\MB_N$ by $L_N (T)$. Based on the above expression of $T$, one has $L_N (T)=T_{p_k}$ and $L_M (T)=T_{p_{k+l}}$ unless $k$ or $l$ is equal to zero. 

 Given a multidegree $\mu$ in $S,$ we use $M'$ and $N'$ to denote the smallest elements in $\MB_M$  and $\MB_N$, respectively, with respect to $\succ_{mrlex}$ order dividing $\mu$.
\end{notation}

The two Hasse diagrams given in \Cref{fig:rlex} and \Cref{fig:mrlex} highlight the differences between the $\succ_{rlex}$ order and  $\succ_{mrlex}$ order. 

\begin{figure}[h]
    \centering
    \begin{minipage}{.5\textwidth}
        \centering

  \begin{tikzpicture}
\tikzstyle{every node}=[]
\node[fill=yellow!50, circle, inner sep=0 pt, minimum size=20 pt] (0) at (2,5) {$x_3^2$};
\node[fill=yellow!50, circle, inner sep=0 pt, minimum size=20 pt] (1) at (3,4) {$x_2x_3$};
\node[fill=yellow!50, circle, inner sep=0 pt, minimum size=20 pt] (2) at (3,3) {$x_2^2$};
\node[fill=yellow!50, circle, inner sep=0 pt, minimum size=20 pt] (3) at (5,3) {$x_1x_3$};
\node[fill=yellow!50, circle, inner sep=0 pt, minimum size=20 pt] (4) at (4,2) {$x_1x_2$};
\node[fill=yellow!50, circle, inner sep=0 pt, minimum size=20 pt] (5) at (4,1) {$x_1^2$};
\node[fill=blue!50,circle,inner sep=3 pt] (6) at (4,5) {$x_2x_4$};
\node[fill=blue!50,circle,inner sep=3 pt] (7) at (5,4) {$x_1x_4$};
\node[fill=blue!50,circle,inner sep=3 pt] (8) at (6,5) {$x_1x_5$};
\node[fill=blue!50,circle,inner sep=3 pt] (9) at (5,6) {$x_2x_5$};
\node[left] at (1.west) {$T_4$};
\node[left] at (2.west) {$T_2$};
\node[right] at (3.east) {$T_3$};
\node[right] at (4.east) {$T_1$};
\node[right] at (5.east) {$T_0$};
\node[left] at (0.west) {$T_5$};
\node[left] at (6.west) {$T_7$};
\node[right] at (7.east) {$T_6$};
\node[right] at (8.east) {$T_8$};
\node[right] at (9.east) {$T_9$};

\node at (2,5.6) {$M$};
\node at (5,6.8) {$N$};

\draw (0)--(1)--(2)--(4)--(5);
\draw (6)--(1)--(3)--(4);
\draw (6)--(7)--(3);
\draw (9)--(8)--(7);
\draw (9)--(6);

\end{tikzpicture}
\caption{ $\succ_{rlex}$  order}\label{fig:rlex}
    \end{minipage}%
    \begin{minipage}{0.5\textwidth}
        \centering
  \begin{tikzpicture}
\tikzstyle{every node}=[]
\node[fill=green!30, circle, inner sep=0 pt, minimum size=20 pt] (0) at (2,5) {$x_3^2$};
\node[fill=green!30, circle, inner sep=0 pt, minimum size=20 pt] (1) at (3,4) {$x_2x_3$};
\node[fill=green!30, circle, inner sep=0 pt, minimum size=20 pt] (2) at (3,3) {$x_2^2$};
\node[fill=green!30, circle, inner sep=0 pt, minimum size=20 pt] (3) at (5,3) {$x_1x_3$};
\node[fill=green!30, circle, inner sep=0 pt, minimum size=20 pt] (4) at (4,2) {$x_1x_2$};
\node[fill=green!30, circle, inner sep=0 pt, minimum size=20 pt] (5) at (4,1) {$x_1^2$};
\node[fill=pink!70,circle,inner sep=3 pt] (6) at (4,5) {$x_2x_4$};
\node[fill=pink!70,circle,inner sep=3 pt] (7) at (5,4) {$x_1x_4$};
\node[fill=pink!70,circle,inner sep=3 pt] (8) at (6,5) {$x_1x_5$};
\node[fill=pink!70,circle,inner sep=3 pt] (9) at (5,6) {$x_2x_5$};
\node[left] at (1.west) {$T_8$};
\node[left] at (2.west) {$T_6$};
\node[right] at (3.east) {$T_7$};
\node[right] at (4.east) {$T_5$};
\node[right] at (5.east) {$T_4$};
\node[left] at (0.west) {$T_9$};
\node[left] at (6.west) {$T_1$};
\node[right] at (7.east) {$T_0$};
\node[right] at (8.east) {$T_2$};
\node[right] at (9.east) {$T_3$};

\node at (2,5.6) {$M$};
\node at (5,6.8) {$N$};

\draw (0)--(1)--(2)--(4)--(5);
\draw (6)--(1)--(3)--(4);
\draw (6)--(7)--(3);
\draw (9)--(8)--(7);
\draw (9)--(6);

\end{tikzpicture}
\caption{$\succ_{mrlex}$  order}\label{fig:mrlex}
    \end{minipage}
\end{figure}

The main result of this section describes a Gr\"obner basis of $T(I)$ with respect to  the mixed reverse lexicographic order, $\succ_{mrlex}$, on $R$. 

\begin{theorem}\label{grobnermixedrevlex}
Let $\mathcal{G}_2=\{\underline{T_{u}T_{v}}-T_{u'}T_{v'}: uv=u'v' \text{ and } u,v \succ_{mrlex} v'\}$. Then $\mathcal{G}_2$ is a Gr\"obner basis of $T(I)$ with respect to the mixed reverse lexicographic order $\succ_{mrlex}$.
\end{theorem}

\begin{proof} Analogous to the proof of \Cref{grobnerrevlex}, it suffices to show that, for each multidegree $\mu$, the fiber graph $\Gamma_{\mu}(I)$ with respect to $\mathcal{G}_2$ is empty or it has a unique sink by \Cref{fibergraphdefinesGB}. As discussed in the proof of \Cref{grobnerrevlex}, when $\mu$ has an odd degree, its fiber graph is empty. It is possible that the fiber graph of $\mu$ is empty when its degree is even. In these two cases, there is nothing to prove.  The remaining situation is proved in \Cref{cor:uniqueSink2}.
\end{proof}

The following lemmas are analogous versions of \Cref{fibergraphmovesforrlex} and  \Cref{sinkreductionrlex} in the mixed reverse lexicographic order.

\begin{lemma}\label{fibergraphmovesformrlex} 
Let $\mu$ be a multidegree and  $T$ be a vertex of $\Gamma_{\mu}(I)$.  
 \begin{enumerate}[(a)]
    \item If $L_M (T)$ exists and $L_M (T) \neq T_{M'}$ then there is an edge from $T$ to to a vertex of $\Gamma_{\mu}(I)$, $T'$, such that 
    $$L_M (T) >_{mrlex}  L_M (T') =  T_{M'}.$$
    \item Suppose $L_M (T)$ does not exist. If $L_N (T)$ exists and $L_N (T) \neq T_{N'}$, then there is an edge from $T$ to $T'$ such that 
    $$L_N (T) >_{mrlex}  L_N (T') =  T_{N'}.$$
\end{enumerate}
\end{lemma}

\begin{proof}
The proof of this lemma is analogous to that of  \Cref{fibergraphmovesforrlex} with the following difference. For the proof of part (b), one can find a monomial $u$ such that $T_uL_N(T)$ divides $T$. This monomial $u$ has to be in $\MB_N$ since $L_M (T)$ does not exist. Depending on our conditions, we can create a monomial $u'$ as in the proof of \Cref{fibergraphmovesforrlex} and this monomial $u'$ must be in $\MB_N$. 
\end{proof}

\begin{lemma}\label{sinkreductionmrlex}
Let $\mu$ be a multidegree such that $\Gamma_{\mu}(I)$ is not  empty.    If $\Gamma_{\mu}(I)$ has a vertex of type $N$, then any sink of $\Gamma_{\mu}(I)$ is of type $N$. In particular, each sink is of the form $Z_{N'}T_{N'}$ where $Z_{N'}$ is a sink of $\Gamma_{\frac{\mu}{N'}}(I)$. Otherwise, every sink of $\Gamma_{\mu}(I)$ is of the form $Z_{M'}T_{M'}$ where $Z_{M'}$ is  a sink of $\Gamma_{\frac{\mu}{M'}}(I)$.

\end{lemma}

The proof of \Cref{sinkreductionmrlex} is completely analogous to  \Cref{sinkreductionrlex}. Similar to \Cref{cor:uniqueSink}, we conclude the following  for the mixed revlex order.

\begin{corollary}\label{cor:uniqueSink2} 
For any  multidegree $\mu$ of degree $2r$, where $\Gamma_{\mu}(I)$ is non-empty, $\Gamma_{\mu}(I)$  has a unique sink $T= T_{p_1} \cdots T_{p_r}$ satisfying that $T_{p_i}$ is the smallest variable in mixed revlex order so that $p_i$ divides $\mu_i=\mu /\prod_{j=i+1}^rp_j$.
\end{corollary}

As in the previous subsection, we conclude the current subsection by presenting an explicit description of the relations between the indices of the variables that are involved in the standard factorization of a sink with respect to mixed reverse lexicographic order.

\begin{lemma}\label{TailComparisonMixRevlex}
Let $\Gamma_{\mu}(I)$ be a nonempty fiber graph and $T=T_{p_1}\cdots T_{p_r}$ be its unique sink.  Let $p_{k_1}=x_ix_j$ and $p_{k_2}=x_{i'}x_{j'}$ where $k_1 <k_2$.

\begin{enumerate} 
    \item[(a)] If $p_{k_1}$  and $p_{k_2}$ are both in $\MB_M$ or $\MB_N$, then $i \leq i'$ and $j \leq j'$.
    \item[(b)]  If $p_{k_1} \in \MB_N$  and $p_{k_2} \in \MB_M$, then $i \leq i'$. 
\end{enumerate}
\end{lemma}
\begin{proof} (a) Notice that when  $p_{k_1}$ and $p_{k_2}$ are both in $\MB_N$ or $\MB_M$, $\succ_{rlex}$ and $\succ_{mrlex}$ orders coincide. Thus the proof is obtained from \Cref{TailComparisonRevlex}.

(b) Suppose $p_{k_1} \in \MB_N$  and $p_{k_2} \in \MB_M$.  If $i >i'$, then $x_i x_{j'} \in \MB_M$ and $x_{i'} x_j \in \MB_N$ which can be easily verified by checking the indices. Note that $p_{k_1}, p_{k_2} \succ_{mrlex} x_ix_{j'}$. Then $T_{p_{k_1}} T_{p_{k_2}} - T_{i'j} T_{ij'} \in \mathcal{G}_2$ implying that $T$ is not a sink, a contradiction.
\end{proof}

\begin{remark}
Similar to \Cref{TailComparisonRevlex}, uniqueness of the sink is not necessary for the lemma to hold.  
\end{remark}

\section{Gr\"obner Bases of the Multi-Rees Algebras of Strongly Stable Ideals}\label{sec:6}

In this section, we provide a quadratic Gr\"obner basis for the defining ideal of the multi-Rees algebra $\mathcal{R}(I_1 \oplus \cdots \oplus I_r)$ when $r=2$ and each ideal is a strongly stable ideal with two quadratic Borel generators. As a result of \Cref{multireesisfibertype}, it suffices to find a quadratic Gr\"obner basis for the toric ideal associated to its special fiber. Similar to \Cref{sec:5}, we follow an approach that involves fiber graphs of multidegrees with respect to a collection of marked binomials.

We begin by setting our notation and introducing our objects. 
 
\begin{notation}\label{I1I2}
We set $I_1=\MB (M_1,N_1)$ and $I_2=\MB(M_2,N_2)$ where  $M_i=x_{a_i}x_{b_i}$ and  $N_i=x_{c_i}x_{d_i}$ for $i=1,2$. Note that $c_i<a_i \leq b_i <d_i$ for each $i$. Additionally, we assume $d_1 \leq d_2$. As in the previous section, we denote the list of the minimal monomial generators of $\MB (M_i)$ by  $\MB_{M_i}$ and the list of the minimal monomial generators of $\MB (N_i) \setminus \MB (M_i)$ by $\MB_{N_i}$ for each $i$.
\end{notation}

\begin{example}
 \Cref{B(M_i)} shows the regions defined in \Cref{I1I2}.
 
 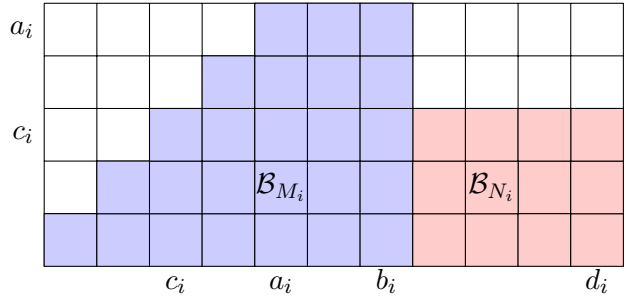
\begin{figure}[h]
\begin{center}
    	\begin{tikzpicture}[scale=0.7]
	
	\begin{scope}[greennode/.style={circle, draw=green!60, fill=green!5}, bluenode/.style={circle, draw=blue!60, fill=blue!5}, rednode/.style={circle, draw=red!60, fill=red!5}]
	\end{scope}
	
\draw[step=1.0,black] (0,0) grid (11,5);

\node[left] at (0,2.5) {$c_i$};
\node[left] at (0,4.5) {$a_i$};

\node[above] at (2.5,-0.7) {$c_i$};
\node[above] at (4.5,-0.7) {$a_i$};
\node[above] at (6.5,-0.7) {$b_i$};
\node[above] at (10.5,-0.7) {$d_i$};

\filldraw[fill=blue!20!white, draw=black] (0,0) rectangle (1,1);
\filldraw[fill=blue!20!white, draw=black] (1,0) rectangle (2,1);
\filldraw[fill=blue!20!white, draw=black] (2,0) rectangle (3,1);
\filldraw[fill=blue!20!white, draw=black] (3,0) rectangle (4,1);
\filldraw[fill=blue!20!white, draw=black] (4,0) rectangle (5,1);
\filldraw[fill=blue!20!white, draw=black] (5,0) rectangle (6,1);
\filldraw[fill=blue!20!white, draw=black] (6,0) rectangle (7,1);
\filldraw[fill=red!20!white, draw=black] (7,0) rectangle (8,1);
\filldraw[fill=red!20!white, draw=black] (8,0) rectangle (9,1);
\filldraw[fill=red!20!white, draw=black] (9,1) rectangle (10,2);
\filldraw[fill=red!20!white, draw=black] (9,0) rectangle (10,1);
\filldraw[fill=blue!20!white, draw=black] (1,1) rectangle (2,2);
\filldraw[fill=blue!20!white, draw=black] (2,1) rectangle (3,2);
\filldraw[fill=blue!20!white, draw=black] (3,1) rectangle (4,2);

\fill[fill=blue!20!white, draw=black] (4,1) rectangle (5,2) (4.5,1.5) node {$\MB_{M_i}$};

\filldraw[fill=blue!20!white, draw=black] (5,1) rectangle (6,2);

\filldraw[fill=red!20!white, draw=black] (10,0) rectangle (11,1);
\filldraw[fill=red!20!white, draw=black] (10,1) rectangle (11,2);
\filldraw[fill=red!20!white, draw=black] (10,2) rectangle (11,3);

\filldraw[fill=blue!20!white, draw=black] (6,1) rectangle (7,2);
\filldraw[fill=red!20!white, draw=black] (7,1) rectangle (8,2); 

\filldraw[fill=red!20!white, draw=black] (8,1) rectangle (9,2) (8.5,1.5) node {$\MB_{N_i}$};

\filldraw[fill=red!20!white, draw=black] (7,2) rectangle (8,3); 

\filldraw[fill=red!20!white, draw=black] (8,2) rectangle (9,3);
\filldraw[fill=red!20!white, draw=black] (9,2) rectangle (10,3);

\filldraw[fill=blue!20!white, draw=black] (2,2) rectangle (3,3);
\filldraw[fill=blue!20!white, draw=black] (3,2) rectangle (4,3);
\filldraw[fill=blue!20!white, draw=black] (4,2) rectangle (5,3);
\filldraw[fill=blue!20!white, draw=black] (5,2) rectangle (6,3);
\filldraw[fill=blue!20!white, draw=black] (6,2) rectangle (7,3);
\filldraw[fill=blue!20!white, draw=black] (3,3) rectangle (4,4);
\filldraw[fill=blue!20!white, draw=black] (4,3) rectangle (5,4);
\filldraw[fill=blue!20!white, draw=black] (5,3) rectangle (6,4);
\filldraw[fill=blue!20!white, draw=black] (6,3) rectangle (7,4);
\filldraw[fill=blue!20!white, draw=black] (4,4) rectangle (5,5);
\filldraw[fill=blue!20!white, draw=black] (5,4) rectangle (6,5);
\filldraw[fill=blue!20!white, draw=black] (6,4) rectangle (7,5);

	\end{tikzpicture}

\caption{The regions for $I_i=\mathcal{B}(x_{a_i}x_{b_i},x_{c_i}x_{d_i})$ with $c_i<a_i \leq b_i<d_i$}    
\label{B(M_i)}
    \end{center}
\end{figure}

\end{example}

Recall that the defining ideal of $\mathcal{R}(I_1 \oplus I_2)$ is the kernel of the following $S$-algebra homomorphism given by
\[\begin{array}{rcl} \varphi:S[T_u,Z_v: u \in \mathcal{B}_{M_1}\cup \MB_{N_1}, v \in \mathcal{B}_{M_2}\cup \MB_{N_2}] & \rightarrow & S[t,z]=\mathbb{K}[x_1,\dots, x_n,t,z] \\ T_u & \rightarrow & ut \\ Z_v & \rightarrow & vz, \end{array} \] 
and extended algebraically. The defining ideal of the special fiber $\mathcal{F} (I_1\oplus I_2)$ is the kernel of $\varphi '$, the induced $\mathbb{K}$-algebra homomorphism of $\varphi$,  and it is denoted by  $T(I_1\oplus I_2)$. 

Note that the ring $R=\mathbb{K}[T_u,Z_v:u \in \mathcal{B}_{M_1}\cup \MB_{N_1}, v \in \mathcal{B}_{M_2}\cup \MB_{N_2}] $ inherits the multigrading from $S[t,z]$ where $R= \bigoplus_{\mu} R_{\mu}$ with $\mu$ ranging over monomials of $S[t,z]$ and $R_{\mu}$ being the $\KK$-vector space described as
$$R_{\mu} = \spa_{\mathbb{K}} \{ V=\prod T_u^{a_u} Z_v^{b_v} \in R ~:~ \varphi(V)=\mu \}.$$

In what follows, we introduce a  collection  of marked binomials $\mathcal{G}$  to study fiber graphs of multidegrees with respect to $\mathcal{G}$. In particular, we prove that $\mathcal{G}$ is a Gr\"obner basis of $T(I_1\oplus I_2)$ with respect to the following monomial order on $R$.

\begin{definition}\label{def:02order}
We define the \emph{head and tail order}, denoted by $\succ_{ht}$, on $R$ as follows. We first introduce the following variable order, denoted by $>_{ht}$, on $R$.

\begin{itemize}
    \item If $u,v\in \mathcal{B}_{M_1}\cup \MB_{N_1}$, and $u \succ_{rlex} v$, then $T_u >_{ht} T_{v}$.
    \item If $u \in \mathcal{B}_{M_1}\cup \MB_{N_1}$ and  $v \in \mathcal{B}_{M_2}\cup \MB_{N_2},$ then $T_u >_{ht} Z_v$.
    \item If $u,v \in \mathcal{B}_{M_2}\cup \MB_{N_2}$, and $u \succ_{mrlex} v$, then $Z_u >_{ht} Z_{v}$. 
\end{itemize}
In particular, we order the variables of $R$ in the following way.
\begin{align*}
T_{11} >_{ht}T_{12}>_{ht} T_{22} >_{ht} \dots >_{ht} T_{1b_{1}} >_{ht} T_{2b_{1}} >_{ht} \dots>T_{a_{1}b_{1}}>_{ht}\\
T_{1(b_{1}+1)} >_{ht} \dots >_{ht} T_{c_{1}(b_{1}+1)} >_{ht} \dots >_{ht} T_{1d_{1}} >_{ht} \dots >_{ht} T_{c_{1}d_{1}} >_{ht}\\
Z_{1(b_{2}+1)} >_{ht} \dots >_{ht} Z_{c_{2}(b_{2}+1)} >_{ht} \dots >_{ht} Z_{1d_{2}} >_{ht} \dots >_{ht} Z_{c_{2}d_{2}} >_{ht}\\
Z_{11} >_{ht} Z_{12} >_{ht} Z_{22} >_{ht} \dots >_{ht} Z_{1b_{2}} >_{ht} Z_{2b_{2}} >_{ht} \dots >_{ht} Z_{a_{2}b_{2}}.
\end{align*}
Finally, head and tail order, $\succ_{ht}$, is the reverse lexicographical order on $R$ induced by $>_{ht}$.
\end{definition}

\begin{remark}
The head and tail order, $>_{ht}$, is actually the direct product order of $>_{rlex}$ and $>_{mrlex}$ on $\mathbb{K}[T_u : u \in \mathcal{B}_{M_1}\cup \MB_{N_1}]$ and $\mathbb{K}[Z_v : v \in \mathcal{B}_{M_2}\cup \MB_{N_2}]$ defined in \cite[Section~2.1.2]{EH}.
\end{remark}

\begin{notation}\label{GquadraticTotal}
Let $\mathcal{G}$ be the union of the following collection of marked binomials where 
\begin{align*}
   \mathcal{G}_1 &=\{\underline{T_{u}T_{v}}-T_{u'}T_{v'}: uv=u'v' \text{ and } u,v \succ_{rlex} v'\}, \\
   \mathcal{G}_2 & =\{\underline{Z_{u}Z_{v}}-Z_{u'}Z_{v'}: uv=u'v' \text{ and } u,v \succ_{mrlex} v'\},\\
   \mathcal{G}_3 &=\{\underline{T_{u}Z_{v}}-T_{u'}Z_{v'}: uv=u'v' \text{ and } v \succ_{mrlex} v'\}.
\end{align*}
Note that $\mathcal{G}_1$ is the Gr\"obner basis of  $I_1$ with respect to term order $\succ_{rlex}$ on $\mathbb{K}[T_u : u \in \mathcal{B}_{M_1}\cup \MB_{N_1}]$ and $\mathcal{G}_2$ is the Gr\"obner basis of  $I_2$ with respect to the mixed reverse lexicographical order, $\succ_{mrlex}$, on $\mathbb{K}[Z_v : v \in \mathcal{B}_{M_2}\cup \MB_{N_2}]$ by \Cref{grobnerrevlex} and \Cref{grobnermixedrevlex}. Marked terms of each $\mathcal{G}_i$ for $i=1,2,3$ are leading terms with respect to the head and tail order, $\succ_{ht}$, on $R$. 

\end{notation}

\begin{remark}
One can also view $\mathcal{G}$ as the collection of $2 \times 2$ minors of the symmetric ladders where $\mathcal{G}_1$ and $\mathcal{G}_2$ correspond to collection of $2\times 2$ minors of symmetric ladders associated to $I_1$ and $I_2,$ respectively. Since  $\mathcal{G}_3$ is expressed in terms of $T$ and $Z$ variables, we can view it as the  collection of $2 \times 2$ minors of ``stacked" symmetric ladders. Here we use the term stacked to imply that the minors can be read from two different levels where the ladder corresponding to the first ideal is placed on the first level and the second one is placed on the second level while being aligned with respect to their initial entries. 

\end{remark}

As we have seen in \Cref{lowestdivisor} and \Cref{Mixlowestdivisor}, it is useful to set a standard factorization to express each monomial in $R$ with respect to the head and tail order.

\begin{definition}\label{fibergraphmulti2}
Let $V$ be a monomial in $R$. We express  $V$ as 
$V=(T_{m_1}\cdots T_{m_p})(Z_{n_1} \cdots Z_{n_q})$ such that $T_{m_1}  \geq_{ht} \dots \geq_{ht}  T_{m_p} >_{ht} Z_{n_1} \geq_{ht} \dots \geq_{ht} Z_{n_q}$, where $m_i \in \mathcal{B}_{M_1}\cup \MB_{N_1}$ and $n_j \in \mathcal{B}_{M_2}\cup \MB_{N_2}$ for each $i \in [p]$ and  $j \in [q]$.
\end{definition}

One can define fiber graphs of $I_1 \oplus I_2$ at multidegrees as in \Cref{sec:5}. We recall this construction below.

\begin{definition}
Given a multidegree  $\mu \in S[t,z]$, we denote the fiber graph of $I_1 \oplus I_2$ at $\mu$ with respect to $\mathcal{G}$ by $\Gamma_{\mu} (I_1,I_2)$. The vertices of the fiber graph are the monomials of $R_{\mu}$. In other words, a monomial $V=(T_{m_1}\cdots T_{m_p})(Z_{n_1} \cdots Z_{n_q})$ is a vertex if and only if $\mu = (m_1 \cdots m_p) (n_1\cdots n_q)$. Given two vertices $V$ and $V'$, there is a directed edge from $V$ to $V'$ whenever $V \rightarrow_{\mathcal{G}} V'$. 
\end{definition}

\begin{remark}\label{rem:trivialPQ}
Let $V=(T_{m_1}\cdots T_{m_p})(Z_{n_1} \cdots Z_{n_q})$ be a vertex in $\Gamma_{\mu} (I_1,I_2)$.  If $p=0$, then $V$ is a vertex in  $\displaystyle \Gamma_{\mu}(I_2)$  with respect to $\mathcal{G}_2$. Similarly, if $q=0$, then $V$ is a vertex in  $\Gamma_{\mu}(I_1)$  with respect to $\mathcal{G}_1$. 
\end{remark}

\begin{example}\label{SinkGraph}
Let $I_1=\MB (x_4x_5,x_2x_6), I_2=\MB(x_4^2,x_3x_6)$ and $\mu= x_2x_3x_4^2x_5x_6t^2z \in S[t,z].$ Note that
all elements of $R_{\mu}$ are the vertices of $\Gamma_{\mu} (I_1,I_2)$. In \Cref{fig:fibergraph}, we label each edge of the fiber graph with the related reduction producing  it.

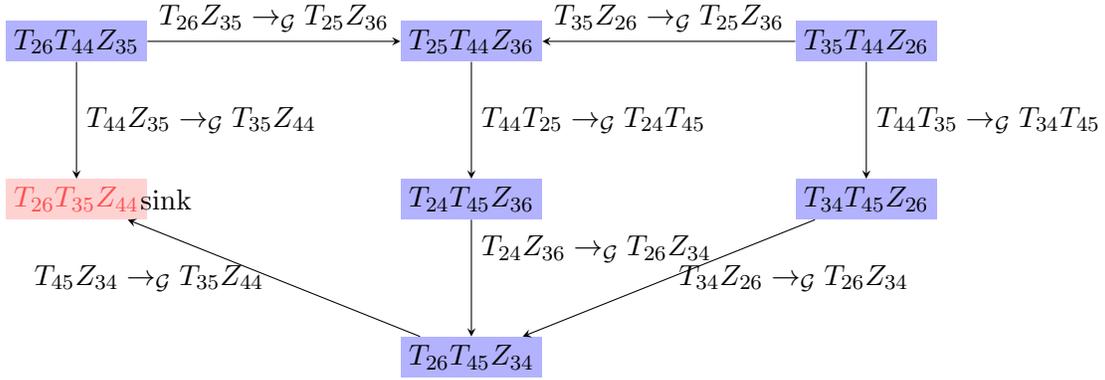
\begin{figure}[hbt]
  \begin{center}
    \begin{tikzpicture}[scale=0.7,->,>=stealth]
      \node[fill=blue!30,rectangle,inner sep=3 pt] (0) at (0,0) {$T_{25}T_{44}Z_{36}$};
      \node[fill=blue!30,rectangle,inner sep=3 pt] (1) at (-7.5,0)  {$T_{26}T_{44}Z_{35}$};
      \node[fill=blue!30,rectangle,inner sep=3 pt](2) at (7.5,0) {$T_{35}T_{44}Z_{26}$};
      \node[fill=blue!30,rectangle,inner sep=3 pt] (3) at (0,-3) {$T_{24}T_{45}Z_{36}$};
      \node[fill=pink!70,rectangle,inner sep=3 pt]  (4) at (-7.5,-3)  {\textcolor{red!70}{{$T_{26}T_{35}Z_{44}$}}};
        \node[fill=blue!30,rectangle,inner sep=3 pt] (5) at (7.5,-3)  {$T_{34}T_{45}Z_{26}$};
         \node[fill=blue!30,rectangle,inner sep=3 pt] (6) at (0,-6)  {$T_{26}T_{45}Z_{34}$};

\node at (-5.8,-3) {sink};
     
    \draw (1)--(0) node[midway,above] {$T_{26}Z_{35} \rightarrow_{\mathcal{G}} T_{25}Z_{36}$} ;
    \draw (1)--(4) node[midway,right] {$T_{44}Z_{35} \rightarrow_{\mathcal{G}} T_{35}Z_{44}$}; 
    \draw (0)--(3) node[midway,right] {$T_{44}T_{25} \rightarrow_{\mathcal{G}} T_{24}T_{45}$}; 
    \draw (2)--(0) node[midway,above] {$T_{35}Z_{26} \rightarrow_{\mathcal{G}}  T_{25}Z_{36}$}; 
    \draw (2)--(5) node[midway,right] {$T_{44}T_{35} \rightarrow_{\mathcal{G}} T_{34}T_{45}$};      
    \draw (3)--(6) node[near start,right] {$T_{24}Z_{36} \rightarrow_{\mathcal{G}} T_{26}Z_{34}$} ;  
    \draw (5)--(6) node[midway,right] {$T_{34}Z_{26} \rightarrow_{\mathcal{G}}  T_{26}Z_{34}$};
     \draw (6)--(4) node[midway,left] {$T_{45}Z_{34} \rightarrow_{\mathcal{G}} T_{35}Z_{44}$};                    
    \end{tikzpicture}
    \caption{The fiber graph  of $I_1 \oplus I_2$ at the  multidegree $x_2x_3x_4^2x_5x_6t^2z$}
\label{fig:fibergraph}
    \end{center}
\end{figure}

\end{example}

We are now ready to state one of the main results of this section.

\begin{theorem}\label{thm:GBs7}
The collection $\mathcal{G}$ is a Gr\"obner basis of $T(I_1 \oplus I_2)$ with respect to the head and tail order, $\succ_{ht}$, on $R$. 
\end{theorem}

\begin{proof}
Analogous to the proofs of \Cref{grobnerrevlex} and \Cref{grobnermixedrevlex}, it suffices to show that for a multidegree $\mu \in S[t,z]$, the fiber graph  $\Gamma_{\mu} (I_1,I_2)$ is empty or has a unique sink by  \Cref{fibergraphdefinesGB} and \Cref{directsumfibergrapgGB}. If the fiber graph is empty, we are done. If not, this statement is proved in \Cref{prop:fiberUniqueSink}.
\end{proof}

In comparison to fiber graphs of $I_1$ and $I_2$, studying sinks of fiber graphs of $I_1 \oplus I_2$ is much more complex because one needs to consider reductions associated to both ideals via the collection $\mathcal{G}_3$ given in terms of $T$ and $Z$ variables. In what follows, we carefully set our notation and establish several auxiliary lemmas to utilize in our proof of \Cref{prop:fiberUniqueSink}. The proofs of Lemmas \ref{atMostj_1}, \ref{lem:obs}, and \ref{atMostc_1} are postponed until the end of the section.

\begin{notation}\label{twoV}
Let  $\mu$ be a multidegree such that $\Gamma_{\mu} (I_1,I_2)$ is nonempty and $p, q \geq 1$. Let $V$ and $V'$  be two vertices of $\Gamma_{\mu} (I_1,I_2)$ such that
\begin{align*}
  V= \underbrace{(T_{m_1}\cdots T_{m_p})}_{V_1} \underbrace{(Z_{n_1} \cdots Z_{n_q})}_{V_2} \text{ and } 
   V'= \underbrace{(T_{m'_1}\cdots T_{m'_p})}_{V'_1} \underbrace{(Z_{n'_1} \cdots Z_{n'_q})}_{V'_2}.
\end{align*}
 Notice that $\mu=\mu_1\mu_2=\mu'_1\mu'_2$ where  $\mu_1= m_1\cdots m_p$, $\mu_2= n_1\cdots n_q$, $\mu'_1= m'_1\cdots m'_p$, and $\mu'_2= n'_1\cdots n'_q$.  We use $\textbf{i}_k, \textbf{j}_k,\textbf{i}'_k, \textbf{j}'_k$ for the indices associated to $m_k$ and $m'_k$, and  $i_l,j_l, i'_l, j'_l$ for the indices associated to $n_l$ and $n'_l$. These two lists of indices should not be confused. More explicitly, for each $k\in  [p]$ and $l\in [q]$, we set
\begin{align*}
    &m_k= x_{\textbf{i}_k} x_{\textbf{j}_k}  \text{ and } n_l= x_{i_l} x_{j_l}\\
    &m'_k= x_{\textbf{i}'_k} x_{\textbf{j}'_k}  \text{ and } n'_l= x_{i'_l} x_{j'_l}.
\end{align*}

 We further introduce the following lists in which an element can appear more than once.
\begin{align*}
   \mathcal{V} = \mathcal{V}_1 \cup \mathcal{V}_2 \text{ where }  &\mathcal{V}_1= \{ \textbf{i}_k, \textbf{j}_k: k\in [p] \} \text{ and } \mathcal{V}_2= \{ i_l,j_l:  l\in [q] \}\\
    \mathcal{V}' = \mathcal{V}'_1 \cup \mathcal{V}'_2 \text{ where }  &\mathcal{V}'_1= \{ \textbf{i}'_k, \textbf{j}'_k: k\in [p] \} \text{ and } \mathcal{V}'_2= \{ i'_l,j'_l:  l\in [q] \}
\end{align*}

 Note that the two lists $\mathcal{V}$ and $\mathcal{V}'$ coincide  as $\mu_1 \mu_2=\mu= \mu'_1 \mu'_2$.  In addition, we denote the first and second half of the indices associated to $n_1, \ldots, n_q$ by
 $$\mathcal{I}= \{i_1, \ldots, i_q\} \text{ and } \mathcal{J}= \{j_1, \ldots, j_q\}.$$ 
 
 It is clear that $\mathcal{V}_2 =\mathcal{I} \cup \mathcal{J}$. 
 \end{notation}
 
  \begin{example} We use \Cref{SinkGraph} to illustrate \Cref{twoV}. Let $\mu=(x_2x_6)(x_4x_5)(x_3x_4)=(x_3x_5)(x_4x_4)(x_2x_6)$, then $V=T_{m_1}T_{m_2}Z_{n_1}$ and $V'=T_{m_1'}T_{m_2'}Z_{n_1'}$  where $m_1=x_2x_6$, $m_2=x_4x_5$, $n_1=x_3x_4$, $m_1'=x_3x_5$, $m_2'=x_4x_4$, and $n_1'=x_2x_6$. Moreover,  $\mathcal{V}_1=\{2,6,4,5\}$, $\mathcal{V}_2=\{3,4\}$, $\mathcal{V}'_1=\{3,5,4,4\}$, $\mathcal{V}'_2=\{2,6\}$, $\mathcal{I}=\{3\}$, and $\mathcal{J}=\{4\}$.
 \end{example}

 For the remainder of this section, we adopt \Cref{twoV}.

\begin{remark}
If $V$ is a sink of $\Gamma_{\mu} (I_1,I_2)$ with respect to $\mathcal{G}$, then $V_1$ must be the unique sink of $\Gamma_{\mu_1} (I_1)$ with respect to $\mathcal{G}_1$ and $V_2$ must be the  unique sink of  $\Gamma_{\mu_2} (I_2)$ with respect to $\mathcal{G}_2$.
\end{remark}

Given a vertex $V$ as in \Cref{twoV}, its associated monomials $m_1, \cdots , m_p$ belong to either the $\mathcal{B}_{M_1}$ region or the $\mathcal{B}_{N_1}$
region and the monomials $n_1, \cdots, n_q$ belong to the $\mathcal{B}_{M_2}$ region or the $\mathcal{B}_{N_2}$ region shown in \Cref{B(M_i)}.
If $V$ is a sink, by \Cref{TailComparisonRevlex} and \Cref{TailComparisonMixRevlex},  we can deduce several relations between the elements of $\mathcal{V}_i$ based on the regions of the associated monomials of $V_i$ for each $i=1,2$.  In the following lemmas, we describe several relations between the elements $\mathcal{V}_1$ and $\mathcal{V}_2$ for a sink $V$ by taking into account the regions of certain monomials.

\begin{lemma}\label{atMostj_1}
Suppose $V$ is a sink of $\Gamma_{\mu} (I_1,I_2)$ such that $n_q \in \MB_{N_2}$. Then 
\begin{enumerate} 
\item[(a)] \label{N_2} each of $n_1, \ldots,n_q $ is contained in $ \MB_{N_2}$  and
    \begin{align}\label{eq:1}
          i_1\leq i_2 \leq \cdots \leq i_q \leq c_2 < b_2 < j_1 \leq j_2 \leq \cdots \leq j_q \leq d_2,
    \end{align} where $b_2,c_2,d_2$ are as defined in \Cref{I1I2}.
\item[(b)] each \emph{$\textbf{j}_k $} is at most $j_1$ for $k\in [p]$, \vskip0.1cm
\item[(c)] each \emph{$\textbf{i}_k$} is either at most $i_1$ or greater than $c_2$ for  $k\in [p]$.
\end{enumerate}

\end{lemma}

Under certain assumptions, we can relate $V$ and $V'$ through the lists given in \Cref{twoV} and counting arguments. In the following set of observations, we provide complementary facts which will be quite useful in our study of sinks of $\Gamma_{\mu} (I_1,I_2)$.

\begin{observation}\label{obs:1}
Suppose $n_q \in \MB_{N_2}$ and $n'_q \in \MB_{M_2}$. One can decompose $\mathcal{V}_2'$  into the following disjoint lists.  
     $$\mathcal{N} = \{ v \in \mathcal{V}_2' : v >b_2\} \text{ and } \mathcal{N}^c = \{ v \in \mathcal{V}_2' : v \leq b_2\}$$ 
     \begin{enumerate}
         \item[(i)]  Observe that $\mathcal{N} \subseteq \{ j'_1, \ldots, j'_{q-1}\} \text{ and } \{ i'_1, \ldots, i'_q, j'_q\} \subseteq \mathcal{N}^c $ since $ j'_q \leq b_2$ and $i'_l \leq a_2$  for each $l \in [q]$. Then  $\mathcal{N}$ has at most $q-1$ elements and  $\mathcal{N}^c $ has at least $q+1$ elements.  \vskip0.1cm
         \item[(ii)]      Let $j_r$ be the smallest element in $\mathcal{J}$ such that $j_r$ is not contained in $\mathcal{V}'_2$. 
         The existence of such an element
is guaranteed by observing that $\mathcal{V}_2' \cap \mathcal{J} \subseteq \mathcal{N}$ by \Cref{atMostj_1}; by (i) there are at most $q-1$ such elements.
Then, there exists $h \in [p]$ such that $x_{j_r}$ divides $m'_h$. As a result, we have $j_1 \leq j_r \leq d_1$  where the first inequality is due to (\ref{eq:1}).  \vskip0.1cm
         \item[(iii)]    Since at most $q$ many of the elements of $\mathcal{N}^c $ can be contained in $\mathcal{I}$ and  none of them is in $\mathcal{J}$ by \Cref{atMostj_1}, there exists $f \in [p]$ such that  $m_f$ is divisible by $x_e$ for $e \in \mathcal{N}^c $, i.e., $e \leq b_2$. Note that $V_1$ is divisible by $T_{m_f}$. For the remainder of this section, we shall denote $m_f= x_e x_{e'}$ by avoiding our standard notation as it is possible to have $e'<e$.  
     \end{enumerate}
\end{observation}

Assuming the existence of another sink vertex, we can provide further relations between the elements of $\mathcal{V}_1$ and $\mathcal{V}_2$. 
These relations will be used in our proof of \Cref{prop:fiberUniqueSink}.

\begin{lemma}\label{lem:obs}
Suppose $V$ and $V'$ are sinks of   $\Gamma_{\mu} (I_1,I_2)$ such that $n_q \in \MB_{N_2} $ and $n'_q \in \MB_{M_2}$. 
\begin{enumerate}
    \item[(a)] \label{NotBc_1b_2}  None of the monomials among $m_1, \ldots, m_p$  is contained in $ \MB (x_{c_1}x_{b_2})$.
    \item[(b)] \label{lem:j_1<a_1} If  the index \emph{$ \textbf{j}_1$}  is greater than $b_2$, then \emph{$\textbf{j}_1$} is at most $a_1$.
    \item[(c)] \label{j_1>b_1} The index $j_1$ is greater than $b_1$.
\end{enumerate}
\end{lemma}

\begin{remark}\label{rem:lemma}
Suppose $V$ and $V'$ are sinks of   $\Gamma_{\mu} (I_1,I_2)$ such that $n_q \in \MB_{N_2} $ and $n'_q \in \MB_{M_2}$. In the light of \Cref{lem:obs}, we can better understand the indices of the monomials from \Cref{obs:1} (ii) and (iii).  
\begin{enumerate}
    \item[(a)]  It follows from   \Cref{lem:obs} (a) that $c_1<e'$ for the monomial $m_f=x_ex_{e'}$ from \Cref{obs:1} (iii). In addition, suppose $ b_2 <\textbf{j}_1$; then we must have $b_2 < \textbf{j}_f$ by (\ref{eq:1}). Since $e \leq b_2$, it follows that $e<e'$ and hence $e=\textbf{i}_f$ and $e' =\textbf{j}_f$.
    \vskip0.1cm
    \item[(b)] \label{j_rm_f}  Recall from \Cref{obs:1} (ii)  that $j_r $ belongs to $\{\textbf{i}'_h, \textbf{j}'_h\}$. Note that $\textbf{i}'_h $ can not be greater than $c_1$. If $c_1<\textbf{i}'_h$, then the monomial $m'_h $ is contained in $\MB_{M_1} \setminus \MB_{N_1}$ such that $\textbf{i}'_h  \leq \textbf{j}'_h \leq b_1$. Since $b_1< j_1 \leq j_r$ by \Cref{lem:obs} (c) and  (\ref{eq:1}), we conclude that  $j_r \notin \{\textbf{i}'_h,\textbf{j}'_h\}$, a contradiction.  Therefore,  $\textbf{i}'_h \leq c_1$ and $j_r= \textbf{j}'_h$.
\end{enumerate}
\end{remark}

 Before the proof of \Cref{prop:fiberUniqueSink}, we present our last auxiliary lemma.  
  
\begin{lemma}\label{atMostc_1}
Suppose $V$ and $V'$ are sinks of   $\Gamma_{\mu} (I_1,I_2)$ with  $n_q \in \MB_{N_2} $ and $n'_q \in \MB_{M_2}$. Then  \emph{$\textbf{i}_1 $} is at most $ c_1$ or one of the monomials from $m_1, \ldots, m_p$ is in $\MB_{N_1}$.
\end{lemma}

\begin{proposition}\label{prop:fiberUniqueSink}
Every nonempty $\Gamma_{\mu} (I_1,I_2)$ with respect to $\mathcal{G}$ has a unique sink.
\end{proposition}
\begin{proof}
Let $\mu$ be a multidegree in $S[t,z]$ such that $\Gamma_{\mu} (I_1,I_2)$ is nonempty. Let 
$V$  be a vertex in $\Gamma_{\mu} (I_1,I_2)$ of the form 
$V= (T_{m_1}\cdots T_{m_p})(Z_{n_1} \cdots Z_{n_q})$. If $p=0$ or $q=0,$ the fiber graph $\Gamma_{\mu} (I_1,I_2)$ becomes a fiber graph of $I_1$ or $I_2$ as discussed in \Cref{rem:trivialPQ}. It follows from \Cref{cor:uniqueSink}  and \Cref{cor:uniqueSink2} that both fiber graphs have unique sinks. For the remainder of the proof, we may assume that $p,q\geq 1$. 

We proceed by contradiction. Suppose $\Gamma_{\mu} (I_1,I_2)$ has two different sinks $V$ and $V'$. It is immediate that $\mu_2 \neq \mu'_2$ following  \Cref{twoV}. Otherwise, we must have $\mu_1 = \mu'_1$ which in turn implies that $V=V'$, a contradiction. Furthermore, we may assume that $n_q \neq  n'_q$. Otherwise, we can consider the multidegree $\mu/ n_q$;  if $n_{q-1} = n'_{q-1}$, consider  the multidegree $\mu/ n_qn_{q-1}$; and so on. Without loss of generality, let $n_q \succ_{mrlex} n'_q$.  Rest of the proof is established by considering all possible regions where $n_q$ and $n'_q$ belong to in $\MB (M_2,N_2)$. In particular, we collect all possibilities under two cases and show that neither of them is possible. Hence, the fiber graph $\Gamma_{\mu} (I_1, I_2)$ must have a unique sink.

\textbf{Case I:} \emph{Suppose $n_q=x_ix_j$ and $n'_q=x_{i'}x_{j'}$ are both in $\MB_{M_2}$ or $\MB_{N_2}$.} 

We first consider the case when they are both in  $\MB_{N_2}$. It follows from our assumption $n_q \succ_{mrlex} n'_q$ that $j<j'$ or $i<i'$. Suppose $j<j'$. Since $i\leq j$, $x_{j'}$ does not divide $n_q$. Moreover, there exists no $Z_{n_l}$ in the support of $V$ such that $x_{j'}$ divides $n_l$ for $l\in [ q-1]$.  Otherwise, $n_q \succ_{mrlex} n_l$ contradicts the definition of $n_q$. Thus, $\mu_2$ is not divisible by $x_{j'}$ while $\mu$ is. Then there exists  $T_{m_l}$ in the support of $V$ such that $x_{j'}$ divides $m_l$. Since $m_l \in \MB (M_1,N_1)$  is divisible by $x_{j'}$ while $j<j'$, the monomial $m' := x_j (m_l/x_{j'})  \in \MB (M_1,N_1)$. In addition, it is immediate that $x_i x_{j'} \in \MB_{N_2}$ and $n_q \succ_{mrlex} x_ix_{j'}$. Hence,  $T_{m_l} Z_{n_q} - T_{m'} Z_{ij'} \in \mathcal{G}$, a contradiction. Then, we must have $j'\leq j$. But $n_q \succ_{mrlex} n_q'$. Thus, we conclude that $j=j'$.

Suppose $i<i'$. Note that $n_l \in \MB_{N_2}$ for each $n_1, \ldots, n_q$ from \Cref{N_2} (a).  Then by \Cref{TailComparisonMixRevlex} and the fact that $i< i'\leq c_2$, there exists no $n_l$ such that $x_{i'}$ divides $n_l$. Thus, $\mu_2$ is not divisible by $x_{i'}$. Since $\mu$ is divisible by $x_{i'}$, there exists $T_{m_l}$ in the support of $V$ such that $x_{i'}$ divides $m_l$. As in the previous paragraph, by letting $m':= x_i (m_l/x_{i'})$, we have $m' \in \MB (M_1,N_1)$.  Hence, $T_{m_l} Z_{n_q} - T_{m'} Z_{i'j} \in \mathcal{G}$, a contradiction.  Therefore, this case is not possible. 

Investigation of the remaining case yields a contradiction, as well. The arguments in
this step are very similar to the previous two paragraphs, so we omit the details.

\textbf{Case II:} \emph{Suppose $n_q \in \MB_{N_2} $ and $n'_q \in \MB_{M_2}$.}

Based on \Cref{atMostc_1}, one needs to consider two subcases to investigate whether the fiber graph can have two sinks. By employing our auxiliary lemmas, we provide a contradiction for each of the two cases and it completes the proof.

\textbf{Subcase 1:} \emph{Suppose} $\textbf{i}_1$ \emph{is at most}  $c_1$.

First note that $\textbf{j}_1>b_2$ because $m_1= x_{\textbf{i}_1} x_{\textbf{j}_1}$ is not in $\MB (x_{c_1}x_{b_2})$ by \Cref{lem:obs} (a). Then  $\textbf{j}_k >b_2$ for each $k \in [p]$ since $\textbf{j}_1$ is the smallest element in $\{\textbf{j}_1,\ldots, \textbf{j}_p\}.$ By \Cref{rem:lemma} (a), 
$e' = \textbf{j}_f$ and $e = \textbf{i}_f\leq a_1$  where $m_f$ is the monomial from \Cref{obs:1} (iii). Additionally, the monomial  $m_1$ is in $ \MB_{M_1}$ because $ \textbf{j}_1 \leq a_1$ by \Cref{lem:obs} (b) and all first indices of monomials in $\MB(M_1,N_1)$ are at most $a_1$.

In what follows, we conclude that $c_1 < \textbf{j}_1$. On the contrary, suppose $\textbf{j}_1 \leq c_1$. Then, we have $x_{\textbf{i}_1}x_{\textbf{i}_f}  \in \MB_{M_1}$ and $x_{\textbf{j}_1}x_{\textbf{j}_f} \in \MB (M_1, N_1)$ by checking the indices; moreover,  $m_1\succ_{rlex} m_f\succ_{rlex} x_{\textbf{j}_1}x_{\textbf{j}_f}$ because $e=\textbf{i}_f\leq b_2<\textbf{j}_1$. Thus, $T_{m_1} T_{m_f} - T_{\textbf{i}_1\textbf{i}_f}T_{\textbf{j}_1\textbf{j}_f} \in \mathcal{G}$, a contradiction.  

In this step, we show that if  $\textbf{i}_k \leq c_1$ for some $\textbf{i}_k$ in $\{\textbf{i}_2, \ldots, \textbf{i}_p\}$, then $ m_k \in \MB_{N_1}$. Otherwise, if $ m_k \in \MB_{M_1}$, we have $\textbf{j}_1 \leq a_1$
(from two paragraphs back) and $\textbf{j}_k \leq b_1$ so $x_{ \textbf{j}_1}x_{ \textbf{j}_k} \in \MB_{M_1}$. 
Additionally, $x_{ \textbf{i}_1}x_{ \textbf{i}_k} \in  \MB_{M_1}$  because $m_k \in B_{M_1}$ and $\textbf{i}_k \leq \textbf{j}_k\leq b_1$. We have $\textbf{i}_k \leq c_1$ (by assumption) and $c_1 < \textbf{j}_1$ (by the previous paragraph), so $\textbf{i}_k < \textbf{j}_1$. Thus $m_1 \succ_{rlex}  m_k \succ_{rlex}  x_{ \textbf{j}_1}x_{ \textbf{j}_k}$ and $T_{m_1} T_{m_k} - T_{\textbf{i}_1 \textbf{i}_k}T_{\textbf{j}_1 \textbf{j}_k} \in \mathcal{G}$, a contradiction.

 In the next step, our goal is to show that there exists $t>1$ such that $m_t$ in $\MB_{M_1}$ when $p\geq 2$. Recall that $p$ is the number of $T$ variables in the support of a vertex of $\Gamma_{\mu}(I_1,I_2)$ as in \Cref{twoV}.  On the contrary, suppose all the monomials in $\{m_2, \ldots, m_p\}$ are in $\MB_{N_1}$.  Let $\mathcal{M}'=\{ \textbf{j}_2, \ldots, \textbf{j}_p, j_1, \ldots, j_q\}\subseteq \mathcal{V}_1 \cup \mathcal{V}_2$. Notice that each element of this list is greater than $b_1$ by our assumption that each $m_k \in \MB_{N_1}$ and \Cref{lem:obs} (c). Furthermore, each of its elements is greater than $b_2$  since $b_2 < \textbf{j}_1$ by the  first step and \Cref{atMostj_1} (a). Note that $\mathcal{V}'_1$ can have at most $p$ elements which are greater than $b_1$. In addition, $\mathcal{V}'_2$ can have at most $q-1$ such elements which are greater than $b_2$ by \Cref{obs:1} (i). Since $\mathcal{M}'$ have $p+q-1$ elements and $\mathcal{V}=\mathcal{V}'$, the list $\mathcal{V}'_1$ has exactly $p$ elements from $\mathcal{M}'$ and the remaining $q-1$ elements of $\mathcal{M}'$ must be in $\mathcal{V}'_2$. In particular, it means $\mathcal{M}'=\{ \textbf{j}'_1, \ldots, \textbf{j}'_p, j'_1, \ldots, j'_{q-1}\}$  since $n'_q \in \MB_{M_2}$. This implies that each $m'_k \in \MB_{N_1}$ and each $n'_l $ belongs to $\MB_{N_2}$ for  $k \in [p]$ and $l \in [q-1]$. Then $\mathcal{V}' \setminus \mathcal{M}'$ contains elements which are either at most $c_1$ or $b_2$. Moreover, this list must contain $\textbf{j}_1$ since   $\mathcal{V}=\mathcal{V}'$. This is not possible because $\textbf{j}_1$ is greater than both $c_1$ and $b_2$ from  the first two steps of this subcase. Therefore, there exists $t>1$ such that $m_t \in \MB_{M_1}$ when $p\geq 2$.

It is worth noting that $\textbf{j}_1 \leq \textbf{i}_t$ for each  $m_t \in \MB_{M_1}$ where $t >1$. Otherwise, both monomials  $x_{\textbf{i}_1} x_{\textbf{i}_t} $  and $x_{\textbf{j}_1} x_{\textbf{j}_t}$ are contained   in $\MB_{M_1}$. In addition, $m_1 \succ_{rlex} m_t \succ_{rlex} x_{\textbf{j}_1} x_{\textbf{j}_t}$. Thus, the binomial $T_{m_1} T_{m_t} - T_{\textbf{i}_1\textbf{i}_t }T_{\textbf{j}_1\textbf{j}_t } $ belongs to $ \mathcal{G}$, a contradiction.

Let $m_1, \ldots, m_s \in \MB_{M_1}$ and $m_{s+1}, \ldots, m_p \in \MB_{N_1}$. Then, by making use of \Cref{TailComparisonRevlex}, \Cref{atMostj_1}, and the main conclusions of each step of this subcase, we  obtain the following chain of inequalities.
\begin{align*}
 \textbf{i}_1 \leq  &\textbf{i}_{s+1} \leq \cdots \leq  \textbf{i}_{p} \leq c_1< \textbf{j}_{1}  \leq \textbf{i}_2 \leq \cdots \leq \textbf{i}_s \leq a_1\\
c_1, b_2 < \textbf{j}_1& \leq \cdots \leq \textbf{j}_s \leq b_1 < \textbf{j}_{s+1} \leq \cdots \leq \textbf{j}_p \leq j_1 \leq \cdots \leq j_q
\end{align*}
Consider the list $\mathcal{O} = \{  \textbf{i}_2, \ldots,  \textbf{i}_s,  \textbf{j}_1, \ldots,  \textbf{j}_p, j_1, \ldots, j_q\} \subseteq \mathcal{V}= \mathcal{V}'_1 \cup \mathcal{V}'_2$ with $p+q+s-1$ elements. Note that each element of $\mathcal{O}$ is greater than both $c_1$ and $b_2$. Then, $\mathcal{V}'_1$ must contain at least $p+s$ many elements from $\mathcal{O}$ since $\mathcal{V}'_2$ can have at most $q-1$ elements from $\mathcal{O}$ by \Cref{obs:1} (i). It follows from the pigeonhole principle that there exists $m'_{l_1} \succ_{rlex} \cdots \succ_{rlex} m'_{l_s}$ such that  $\textbf{i}'_{l_k}, \textbf{j}'_{l_k} \in \mathcal{O}$ for each $k\in [s]$. Since each $\textbf{i}'_{l_k} \in \mathcal{O}$ is greater than $c_1$, each monomial  $m'_{l_k}$ must be contained in $ \MB (M_1) \setminus \MB (N_1)$, in other words, $\textbf{i}'_{l_k} \leq \textbf{j}'_{l_k}\leq b_1$ for $k \in [s]$.  As a result, the list $\mathcal{O}$ has at least $2s$ elements which are at most $b_1$. However, the only such elements of $\mathcal{O}$ are $\{  \textbf{i}_2, \ldots,  \textbf{i}_s,  \textbf{j}_1, \ldots,  \textbf{j}_s\}$, a contradiction.  As a result, Subcase 1 is eliminated. 

\textbf{Subcase 2:} \emph{Suppose} $\textbf{i}_1>c_1$ \emph{and one of the monomials from} $m_2, \ldots, m_p$ \emph{is in} $\MB_{N_1}$.
 
 Let $m_1, \ldots, m_s \in \MB_{M_1}$ and $m_{s+1}, \ldots, m_p \in \MB_{N_1}$ where $s\geq 1$. Then, we  obtain the following chain of inequalities  by making use of   \Cref{TailComparisonRevlex}, \Cref{N_2} (a) and (b), \Cref{NotBc_1b_2} (a) and (c).
\begin{align}\label{ex:3}  
\textbf{i}_{s+1} \leq \cdots \leq  \textbf{i}_{p} \leq c_1< \textbf{i}_{1} \leq \cdots \leq \textbf{i}_s \leq a_1 \notag &\text{ and } c_1 <  \textbf{j}_1 \leq \cdots \leq \textbf{j}_s \leq b_1 \notag\\
b_1, b_2 < \textbf{j}_{s+1}  \leq \cdots \leq \textbf{j}_p \leq j_1 \leq \cdots \leq j_q &\text{ and }
i_1 \leq \cdots \leq i_q \leq c_2
\end{align}
 Consider the list $\mathcal{O} = \{  \textbf{j}_{s+1}, \ldots,  \textbf{j}_p, j_1, \ldots, j_q\} \subseteq \mathcal{V}$ with $p+q-s$ elements.  Note that $\mathcal{O} $ is the collection of all elements of $\mathcal{V}$ which are greater than both $b_1$ and $b_2$. Recall from  \Cref{obs:1} (i) that $\mathcal{V}'_2$ can have at most $q-1$ elements from $\mathcal{O} $. Thus,   the list $\mathcal{V}'_1$ must contain at least $p-s+1$ many elements from $\mathcal{O} $. Suppose $\mathcal{V}'_1$ contains exactly $\tilde{p}:= p+t-s+1$  many of the elements of $\mathcal{O}$ and the remaining $\tilde{q}:=q-t-1$ elements of $\mathcal{O}$ are contained in $\mathcal{V}'_2$ where $0 \leq t \leq q-1$.
 Then there exists $\textbf{i}'_{l_k}, \textbf{j}'_{l_k} \in \mathcal{V}'_1$  such that each $\textbf{j}'_{l_k} \in \mathcal{O} $ for $k\in [\tilde{p}]$. Since elements of $\mathcal{O}$ are greater than $b_1$, we must have $\textbf{i}'_{l_k} \leq  c_1$  for each $k\in [\tilde{p}]$. It means that $\mathcal{V}$ has at least $p-s+1$ many elements which are less than $c_1$. In what follows, we show that $\mathcal{V}$  has only $p-s$ such elements and it yields to a contradiction. Hence, Subcase 2 is eliminated and, as a result,  the fiber graph can not have two sinks satisfying the statement of Case II. 
 
Let $\mathcal{C}$ be the collection of all elements of  $\mathcal{V}$ which are less than $c_1$. It follows from the  chain of inequalities given in (\ref{ex:3}) that $ \mathcal{C} \subseteq \{ \textbf{i}_{s+1}, \ldots, \textbf{i}_p, i_1, \ldots, i_q \}$. We claim that $\mathcal{C}= \{ \textbf{i}_{s+1}, \ldots, \textbf{i}_p\}$, in other words,  $\mathcal{C}$ has $p-s$ elements. In order to prove the claim, it suffices to show  that $c_1 < i_1$. 

Let $\mathcal{S} =\{\textbf{i}_k, \textbf{j}_k : k \in [s]\} \subseteq \mathcal{V} = \mathcal{V}'_1 \cup \mathcal{V}'_2$ and set $\mathcal{S}_i =\{ u \in \mathcal{V}'_i : u \in \mathcal{S}\}$ for $i=1,2$. Note that none of the elements of $ \mathcal{S}_1$ is equal to  $\textbf{i}'_{l_k}$ or $ \textbf{j}'_{l_k}$  because $\textbf{i}'_{l_k} \leq c_1$ and $b_1<\textbf{j}'_{l_k} $ for $k \in [\tilde{p}]$. Thus,  the list $ \mathcal{S}_1$  is contained in $\mathcal{V}'_1 \setminus \{ \textbf{i}'_{l_k}, \textbf{j}'_{l_k} : k \in [\tilde{p}]\}$. This means $ \mathcal{S}_1$ has at most $2s-2t-2$  elements while  $ \mathcal{S}_2$ has at least $2t+2$  elements.  For each $u \in \mathcal{S}_2$, we have either $u \leq c_2$ or $u > c_2$. If $u \leq c_2$ for some $u \in \mathcal{S}_2,$ then $\textbf{i}_1 \leq u \leq c_2$. It follows from \Cref{atMostj_1}  (c) that $c_1 < \textbf{i}_1 \leq i_1$, proving the claim. 

Suppose each elements of $\mathcal{S}_2$ are greater than $ c_2$. Recall that $\mathcal{V}'_2$ has exactly $\tilde{q}$ elements from $\mathcal{O}$; there exists $i'_{k_l}, j'_{k_l} \in \mathcal{V}'_2$ such that each $j'_{k_l} \in \mathcal{O}$ for $l \in [\tilde{q}]$. Since each element of $\mathcal{O}$ is greater than $b_2$,  each $i'_{k_l} \leq c_2$ for  $l \in [\tilde{q}]$ which implies that each $n'_{k_l} \in \MB_{N_2}$. Note that none of the elements of $\mathcal{S}_2$ is equal to $i'_{k_l}$ or $j'_{k_l}$ for $l \in [\tilde{q}]$. Thus, $\mathcal{S}_2$ is contained in $\mathcal{V}'_2 \setminus \{i'_{k_l}, j'_{k_l}:  l \in [\tilde{q}] \}$. Since $\mathcal{S}_2$ has at least $2t+2$ elements while the latter set has exactly $2t+2$ elements, we must have 
\begin{align*}
  \mathcal{S}_2 = \mathcal{V}'_2 \setminus \{i'_{k_l}, j'_{k_l}:  l \in [\tilde{q}] \}  \text{ and } \mathcal{S}_1 = \mathcal{V}'_1 \setminus \{ \textbf{i}'_{l_k}, \textbf{j}'_{l_k} : k \in [\tilde{p}]\}
\end{align*}
where the second equality follows from observing $ \mathcal{S}_1$ must have $2s-2t-2$ elements. In the light of the first equality and the fact that each element of $\mathcal{S}_2$ is greater than $c_2$, the list $\mathcal{S}_2$ contains pairs  $i'_l, j'_l$ where the associated monomial $n'_l \in \MB (M_2 )\setminus \MB (N_2)$.  Recall that for the pairs that are not in $\mathcal{S}_2$, their associated monomials are in $\MB_{N_2}$. Since we use the $>_{mrlex}$ order on the $Z$ variables, we must have $\mathcal{S}_2 = \{ i'_l, j'_l : l \in \{\tilde{q}+1,\dots,q\}\}$ where $i'_l \leq a_2$ and $j'_l \leq b_2.$ 
 
In the final step, our goal is to show that $m_1 \in \MB_{M_2}$. Notice that if $m_1 \in \MB_{M_2}$, then $c_1<i_1$.  Otherwise, we have $n_1 \in \MB_{N_1}$  by (\ref{ex:3}). Then $n_1\succ_{mrlex}m_1$ and $T_{m_1} Z_{n_1} - T_{n_1} Z_{m_1} \in \mathcal{G}$, a contradiction. Consider the list $\mathcal{S}_2$. Since $\textbf{i}_1$ is the smallest element in $\mathcal{S},$ it is clear that $\textbf{i}_1 \leq a_2$ because  $i'_l \leq a_2$ for $i'_l \in \mathcal{S}_2$. Similarly, since $\textbf{j}_1$ is the smallest among $\textbf{j}_1, \ldots, \textbf{j}_s$, if $\mathcal{S}_2$ contains at least one $\textbf{j}_k$ for $k \in [s]$, then $\textbf{j}_1 \leq b_2.$ So the only case that we need to consider is when $\mathcal{S}_2$ does not contain any such element. Note that it happens when $\textbf{j}_k \in \mathcal{S}_1  $ for each $k \in [s]$.  By the pigeonhole principle, there exists $\textbf{i}'_k, \textbf{j}'_k \in \mathcal{S}_1$ such that $\textbf{i}'_k = \textbf{j}_u$ and $\textbf{j}'_k = \textbf{j}_v$ for some $u, v \in [s]$ and it implies $\textbf{j}_1 \leq \textbf{j}_u \leq a_1$.  Observe that $\textbf{j}_1 \leq \textbf{i}_2.$ Otherwise, $m_1\succ_{rlex} m_2 \succ_{rlex}x_{\textbf{j}_1}x_{\textbf{j}_2}$ and $T_{m_1} T_{m_2} - T_{\textbf{i}_1\textbf{i}_2}T_{\textbf{j}_1\textbf{j}_2} \in \mathcal{G},$ a contradiction. Since $\mathcal{S}_2$ has at least two elements, there exists $k>1$ such that $\textbf{i}_k \in \mathcal{S}_2$. Thus $\textbf{j}_1 \leq \textbf{i}_2 \leq b_2.$ Therefore, $m_1 \in \MB_{M_2}$ and it concludes the proof.
\end{proof}

As an immediate corollary of \Cref{thm:GBs7}, we obtain the following result on the Koszulness of multi-Rees algebras.

 \begin{corollary}\label{thm:MultiReesGB}
The multi-Rees algebra, $\mathcal{R} (I_1 \oplus I_2)$, of strongly stable ideals $I_1$ and $I_2$ with two quadric Borel generators is Koszul. In particular, products of powers of $I_1$ and $I_2$ have linear resolutions.
\end{corollary}

\begin{proof}
Recall from \Cref{thm:GBs7} that the toric ideal  $T(I_1 \oplus I_2)$ has a quadric Gr\"obner basis with respect to the head and tail order. It then follows from \Cref{multireesisfibertype} that the defining ideal of $\mathcal{R} (I_1 \oplus I_2)$ has a quadric Gr\"obner basis. Thus $\mathcal{R} (I_1 \oplus I_2)$ is Koszul by \cite[Theorem~6.7]{EH}. The second statement follows from \cite[Theorem~3.4]{BC}.
\end{proof}

We close the paper with the following question inviting interested researchers to study Koszulness of the multi-Rees algebras of strongly stable ideals.

\begin{question}
Let $I_1, \ldots, I_r$ be a collection of strongly stable ideals satisfying one of the conditions given in \Cref{prop:KoszulCases}. 
Is the multi-Rees algebra $\mathcal{R} (I_1 \oplus \cdots \oplus I_r)$ Koszul?
\end{question}

\textbf{Proof of \Cref{atMostj_1}.}
(a) Recall that  $Z_{n_q}$ is the least variable among  $Z_{n_1}, \ldots, Z_{n_q}$ with respect to the head and tail order. Furthermore, a monomial in $B_{N_2}$ is greater than any monomial in $B_{M_2}$ in $\succ_{mrlex}$ order. Hence if the least of the monomials $n_1, \ldots, n_q$ (namely $n_q$) is in $B_{N_2}$, then all of $n_l$ must be in $ \MB_{N_2}$ for $l \in [q]$. Then, we have Inequality (\ref{eq:1}) from \Cref{TailComparisonMixRevlex}.

(b) Note that $\textbf{j}_k \leq \textbf{j}_p$ for each $k\in  [p]$. Thus, it suffices to show that $ \textbf{j}_p \leq j_1$. On the contrary, suppose $ j_1 < \textbf{j}_p $.  Then,  $x_{\textbf{i}_p }x_{j_1} $ is contained in $ \MB (M_1,N_1)$ because $m_p= x_{\textbf{i}_p }x_{\textbf{j}_p } \in \MB (M_1,N_1)$ while $j_1 < \textbf{j}_p $. In addition, observe that $ x_{i_1} x_{\textbf{j}_p }$ is in $\MB_{N_2}$  by (\ref{eq:1}) and our assumption $d_1 \leq d_2$. Note that $n_1\succ_{mrlex} x_{i_1} x_{\textbf{j}_p } $. Hence, the binomial $T_{m_p} Z_{n_1} - T_{\textbf{i}_pj_1} Z_{i_1 \textbf{j}_p}$ belongs to $ \mathcal{G}$, which contradicts to our assumption on $V$ being a sink.  Therefore,  the index $ \textbf{j}_p $ is at most $ j_1$.

(c) Suppose $i_1 < \textbf{i}_k \leq c_2$ for some $k \in [p]$. Since $m_k = x_{\textbf{i}_k } x_{\textbf{j}_k } \in \MB (M_1, N_1)$ and $i_1 < \textbf{i}_k$, we must have $x_{i_1 } x_{\textbf{j}_k } \in \MB (M_1, N_1)$. Furthermore,  the monomial $x_{\textbf{i}_k } x_{j_1} $ is contained in $ \MB_{N_2}$ and  $n_1\succ_{mrlex} x_{\textbf{i}_k } x_{j_1}$. Thus, $T_{m_k} Z_{n_1} - T_{i_1 \textbf{j}_k } Z_{\textbf{i}_k j_1} \in \mathcal{G}$, a contradiction.
\qed

\textbf{Proof of \Cref{lem:obs}.}
We may assume that the sink vertices $V$ and $V'$ are relative prime, otherwise one may factor out the common factors. (a) On the contrary, suppose there exists $k\in [p]$ such that $m_k= x_{\textbf{i}_k} x_{\textbf{j}_k} \in \MB (x_{c_1}x_{b_2})$. Note that the monomial $x_{c_1}x_{b_2}$ is not written according to our standard notation as it is possible to have $b_2 < c_1$ while we set $\textbf{i}_k\leq c_1$ and $\textbf{j}_k \leq b_2$. Consider the index $r$ from \Cref{obs:1} (ii) such that $n_r= x_{i_r} x_{j_r}$ where $i_r \leq c_2$ and $b_2 < j_r \leq d_1$. Then $x_{i_r} x_{\textbf{j}_k} \in \MB_{M_2}$ and $x_{\textbf{i}_k} x_{j_r} \in \MB (M_1,N_1)$ which can be verified by checking the indexes. Note that $n_r\succ_{mrlex}x_{i_r} x_{\textbf{j}_k}$. Thus, the binomial $T_{m_k} Z_{n_r} - T_{\textbf{i}_kj_r } Z_{i_r \textbf{j}_k} $ is contained in $ \mathcal{G}$, a contradiction.

(b) Suppose $a_1<\textbf{j}_1$ and let $\mathcal{J}'=\{\textbf{j}_1, \ldots, \textbf{j}_p\}$. Observe that each element of $\mathcal{J} \cup \mathcal{J}'  \subset \mathcal{V}$ is greater than $a_1$ and $b_2$ by \Cref{atMostj_1} and (\ref{eq:1}).   Thus $\mathcal{V}$ has at least $p+q$ such elements. Recall from \Cref{obs:1} (i) that $\mathcal{N} $ is the list of all such elements in $\mathcal{V}_2'$ and it has at most $q-1$ elements. Since $\mathcal{V}= \mathcal{V}' =  \mathcal{V}'_1 \cup  \mathcal{V}'_2$, at least $p+1$ elements of $\mathcal{J} \cup \mathcal{J}'$ must be contained in  $\mathcal{V}'_1$. It follows from the pigeonhole principle that there exists $k \in [p]$ such that $m'_k = x_{\textbf{i}'_k} x_{\textbf{j}'_k}$ where $\textbf{i}'_k, \textbf{j}'_k \in \mathcal{J} \cup \mathcal{J}'$. Then $a_1 < \textbf{i}'_k\leq  \textbf{j}'_k$ which implies that  $m'_k \notin \MB (M_1, N_1)$, a contradiction. 

(c) On the contrary, suppose $j_1 \leq b_1$. We first claim that $\textbf{j}_k >b_2$ for all $k \in [p]$. Otherwise, suppose there exists  $k \in [p]$ such that $\textbf{j}_k \leq b_2 <j_1$.  Then $x_{i_1} x_{\textbf{j}_k} \in \MB_{M_2}$ as  $n_1=x_{i_1}x_{j_1} \in \MB_{N_2}$. Note that $x_{i_1} x_{\textbf{j}_k}$ comes later than $n_q$ with respect to $\succ_{mrlex}$.  Furthermore, we have $x_{\textbf{i}_k} x_{j_1} \in \MB_{M_1}$ since $j_1 \leq b_1$. Thus, $T_{m_k} Z_{n_1} - T_{\textbf{i}_k j_1 } Z_{i_1 \textbf{j}_k} \in \mathcal{G}$, a contradiction.  Thus, the claim holds. 

It follows from the claim that the index $e$ from \Cref{obs:1} (iii) must belong to the list $\{\textbf{i}_1 , \ldots, \textbf{i}_p\}$ and, as a result, we have $\textbf{i}_1  \leq b_2 < j_1$ by \Cref{I1I2} and \Cref{twoV}.  In addition,  $\textbf{j}_1 \leq a_1$ by part (b)  of this lemma.  Then, the monomial $x_{i_1} x_{\textbf{i}_1} $ is in $ \MB_{M_2}$ because $i_1 \leq c_2$ and $\textbf{i}_1 \leq b_2$; moreover, $x_{\textbf{j}_1} x_{j_1} \in \MB_{M_1}$ as $\textbf{j}_1 \leq a_1$ and $j_1 \leq b_1$. Recall that $n_1 \succ_{mrlex} x_{i_1} x_{\textbf{i}_1}$ as $n_1 \in \MB_{N_2}$ by \Cref{atMostj_1}. Thus, the binomial $T_{m_1} Z_{n_1} - T_{\textbf{j}_1 j_1} Z_{i_1 \textbf{i}_1}$ belongs to $ \mathcal{G}$, a contradiction.   
\qed

\textbf{Proof of \Cref{atMostc_1}.} 
   On the contrary, suppose $c_1 < \textbf{i}_1$ and $m_k \in \MB_{M_1}$ for each $k\in [p]$. It follows from \Cref{TailComparisonRevlex} (a) that 
 \begin{align}\label{eq:2}
 c_1 < \textbf{i}_1 \leq \textbf{i}_2 \leq \cdots \leq \textbf{i}_p \text{ and } \textbf{j}_1 \leq \textbf{j}_2 \leq \cdots \leq \textbf{j}_p \leq b_1 \text{ where } \textbf{i}_k \leq \textbf{j}_k \text{ for } k \in [p].    
 \end{align}
 In order to get a contradiction, our first step is to show that $n_1 \in \MB_{N_1}$.  For this purpose, recall that $e \leq b_2$ where $m_f = x_ex_{e'}$  is the monomial from  \Cref{obs:1} (iii). By using the chain of inequalities (\ref{eq:2}), we conclude that $c_1< b_2$ since $e \in \{\textbf{i}_f, \textbf{j}_f\}$. Additionally, it is clear that  the multidegree $\mu=\mu_1 \mu_2$ is divisible by $x_{\textbf{i}'_h}$ where $m'_h =x_{\textbf{i}'_h}x_{\textbf{j}'_h}$ such that $\textbf{i}'_h \leq c_1$ and $\textbf{j}'_h=j_r$  based on \Cref{rem:lemma} (b).  Then $\textbf{i}'_h$ must be in $ \mathcal{V}_1 \cup \mathcal{V}_2$ where  $\mathcal{V}_2 = \mathcal{I} \cup \mathcal{J}$. Note that $\textbf{i}'_h$ is not contained in $\mathcal{V}_1 = \{ \textbf{i}_k, \textbf{j}_k: k\in [p] \}$ as each element of this list is greater than $c_1$ by (\ref{eq:2}) whereas $\textbf{i}'_h \leq c_1<b_2$; furthermore, it can not be in  $\mathcal{J}$ as all elements of $\mathcal{J}$ are greater than $b_2$ by (\ref{I1I2}). Then the only option is $\textbf{i}'_h \in \mathcal{I}$ which implies that $i_1 \leq \textbf{i}'_h \leq c_1$. Thus, $n_1=x_{i_1} x_{j_1} \in \MB_{N_1}$ by \Cref{lem:obs} (c) and \Cref{obs:1} (ii).

Our next step is to show that $ b_2 < \textbf{j}_1 \leq a_1.$ It suffices to show that $ b_2 < \textbf{j}_1$ by \Cref{lem:obs} (b). First observe that there exists no $k\in [p]$ such that $m_k \in \MB_{M_2}$. Otherwise, $T_{m_k}Z_{n_1} - T_{n_1}Z_{m_k} \in \mathcal{G}$, a contradiction.  Note that $\mathcal{V}'_2$ has at least $q$ many elements which are at most $a_2$ while $\mathcal{V}_2 \setminus \{\textbf{i}'_h\}$ has $q-1$ many such elements.  Since $\mathcal{V}= \mathcal{V}'$, there exists $\textbf{i}_k \in \mathcal{V}_1$ such that $\textbf{i}'_h \neq \textbf{i}_k \leq a_2$. Then it follows from  (\ref{eq:2}) that $\textbf{i}_1 \leq a_2$.  Thus, $\textbf{j}_1> b_2$. If not,  the monomial $m_1= x_{\textbf{i}_1 } x_{\textbf{j}_1 }  $ is contained in $\MB_{M_2}$, a contradiction. 

In the final step of the proof, recall from \Cref{obs:1} (i) that $\mathcal{V}'_2$ has at least $q+1$ elements which are at most $b_2$ and $\mathcal{N}^c$ is the list containing all such elements. Since  $\mathcal{V} =\mathcal{V}'$ and  $\mathcal{V}_2 \setminus \{\textbf{i}'_h\}$ has $q-1$ many such elements, the list $\mathcal{V}_1=\mathcal{I}' \cup \mathcal{J}'$ must contain at least two elements from $\mathcal{N}^c$ where $\mathcal{I}'= \{\textbf{i}_1, \ldots, \textbf{i}_p\}$ and $\mathcal{J}'= \{\textbf{j}_1, \ldots, \textbf{j}_p\}$.  In particular, both such elements of $\mathcal{V}_1$ must be in $\mathcal{I}'$ because each element of $\mathcal{J}'$ is greater than $b_2$ by \Cref{atMostj_1}. Since $e= \textbf{i}_f$ is already one of these two elements, the other element must be $\textbf{i}_g$ for some $g \in [p]$. Let $L= \max \{ f,g\}$. Then $x_{\textbf{i}_1} x_{\textbf{i}_L} \in \MB_{M_1}$  as $m_1 \in \MB_{M_1}$ where $\textbf{i}_L \leq b_2 <\textbf{j}_1$. Furthermore, $x_{\textbf{j}_1} x_{\textbf{j}_L} \in \MB_{M_1}$ since $ \textbf{j}_L \leq b_1$ by (\ref{eq:2}) and $\textbf{j}_1 \leq a_1$ by the previous paragraph. Thus, $m_1 \succeq_{rlex} m_L=x_{\textbf{i}_L} x_{\textbf{j}_L} \succ_{rlex} x_{\textbf{j}_1} x_{\textbf{j}_L}$ and $T_{m_1} T_{m_L} - T_{\textbf{i}_1 \textbf{i}_L} T_{\textbf{j}_1 \textbf{j}_L}\in \mathcal{G},$ contradiction.
\qed

\par
\bibliographystyle{plain}
\bibliography{references}
\end{document}